\documentclass[a4paper]{amsart}
\usepackage{ricky}
\usepackage[all]{xy}

\makeindex

\begin{document}
\title[Eigenvarieties for non-cuspidal modular forms]{Eigenvarieties for non-cuspidal modular forms over certain PEL Shimura varieties}


\author{Riccardo Brasca}
\email{\href{mailto:riccardo.brasca@imj-prg.fr}{riccardo.brasca@imj-prg.fr}}
\urladdr{\url{http://www.imj-prg.fr/~riccardo.brasca/}}
\address{Institut de Math\'ematiques de Jussieu - Paris Rive Gauche\\
Universit\'e Paris Diderot\\
Paris\\
France}
\thanks{The first author was partially supported by the ANR PerCoLaTor: ANR-14-CE25-0002-01.}

\author{Giovanni Rosso}
\email{\href{mailto:giovanni.rosso@concordia.ca}{giovanni.rosso@concordia.ca}}
\urladdr{\url{https://sites.google.com/site/gvnros/}}
\address{Department of Mathematics and Statistics, Concordia University, Montr\'eal, Quebec, Canada}
\thanks{The second author was partially founded by a FWO travel grant V4.260.14N and the Herchel Smith Postdoctoral Fellowship}


\begin{abstract}
Generalising the recent method of Andreatta, Iovita, and Pilloni for cuspidal forms, we construct an eigenvariety for symplectic and unitary groups that parametrises systems of eigenvalues of overconvergent and locally analytic $p$-adic automorphic forms. This is achieved by gluing some intermediates eigenvarieties of a fixed `degree of cuspidality'. The dimension of these eigenvarieties is explicit and depends on the degree of cuspidality, it is maximal for cuspidal forms and it is $1$ for forms that are `not cuspidal at all'. Under mild assumption, we are able to prove a conjecture of Urban about the dimension of the irreducible components of Hansen's eigenvariety in the case of the group $\GSp_4$ over $\Q$.
\end{abstract}

\maketitle

\section*{Introduction}
Let $p$ be a fixed prime number. Since the seminal work of Hida on congruences modulo $p$  between ordinary modular forms, many progress have been made in the study of $p$-adic families of automorphic forms. Hida's techniques have been successfully adapted in many different settings to construct families of ordinary automorphic eigenforms. See for example \cite{TilUrb,Mauger,hida_control}.

On the other hand, the theory for forms which are of finite slope for a certain $\U_p$-operator but not ordinary has viewed less advancements since the foundational work of Coleman \cite{cole_fam}. Recently, Urban in \cite{urban_eigen} and Hansen in \cite{Hansen} (generalising ideas of Ash and Stevens in \cite{AshSt}) have developed a very general theory for families of eigenforms. Their approach is mainly cohomological.

Contrary to Hida's theory, there was no `coherent' approach to eigenvarieties until the recent work of Pilloni \cite{vincent} and Andreatta, Iovita, and Stevens \cite{over}. Their approach has been generalised first to Siegel modular forms in \cite{AIP} and then to general PEL Shimura varieties with non-empty ordinary locus in \cite{pel}. In these papers, the authors deal only with eigenvarieties for cuspidal forms and do not consider families of Eisenstein series. Such families of Eisenstein series have been proven to be very useful in many arithmetic applications, starting from the seminal work of Ribet to more recent applications to Iwasawa theory.

The reason why \cite{AIP} and \cite{pel} only deal with cuspidal forms is that the authors, to build the eigenvariety, use a general machinery due to Buzzard in \cite{buzz_eigen}. This construction has a technical hypothesis (namely the `projectivity' of the space of forms over the weight space, see Subsection~\ref{subsec: eigen machinery} for more details) that they can only show for the space of cuspidal forms.
On the opposite hand, Kisin--Lai \cite{KisinLai} and others constructed long time ago a one-dimensional eigenvarieties using families of Eisenstein series. These two approaches seem orthogonal.

In the paper at hand we give a unitary way of understanding the two constructions: generalising the method of \cite{AIP} and \cite{pel} we construct eigenvarieties for not necessarily cuspidal $p$-adic modular forms for certain Shimura varieties of PEL type. The key point is that in general the space of all forms is not projective, so one can not apply Buzzard's construction directly. We show that, adding some restrictions on the weights, the space of modular forms of a given `degree of cuspidality' (see below for a precise definition) is projective, and in particular we can apply Buzzard's machinery to construct several equidimensional eigenvarieties. We then explain what can be done to glue these eigenvarieties together.

Let us now state more precisely the results of this paper. In order to simplify the notation we consider in this introduction only the Siegel case (i.e. we consider the symplectic group over $\Q$), but our results hold true also in the Hilbert-Siegel case and (with some additional assumptions) in the unitary case. Let $p > 2$ be an odd prime and let $g \geq 2$ be an integer. Let $\mathfrak S$ be the formal Siegel variety of some fixed level outside $p$ (to be precise we should work with the Shimura variety of Iwahoric level at $p$, but we will ignore this issue in the introduction). We write $\mathfrak S^{\rig}$ for its rigid analytic fiber and, if $v \in \Q_{\geq 0}$, we denote by $\mathfrak S(v)^{\rig}$ the strict neighborhood of the ordinary locus defined by the condition that the Hasse invariant has valuation smaller or equal than $v$. We denote with $\mc W_g$ the weight space, that is a rigid analytic space isomorphic to a disjoint union of $g$-dimensional open balls of radius $1$. We will work with certain subspaces $\mc W(w)_g \subseteq \mc W_g$, that parametrise $w$-analytic weights, where $w$ is a rational. Let $\mc U = \Spm(A) \subset \mc W(w)_g$ be an affinoid, with associated universal character $\chi_{\mc U}^{\un}$. One of the main results of \cite{AIP} is the construction of a sheaf $\underline \omega_{v,w}^{\dagger \chi_{\mc U}^{\un}}$ on $\mathfrak S(v)^{\rig} \times \mc U$ (where $v$ is small enough) that interpolates the usual modular sheaves given by integral weights in $\mc U$. The sheaf extends to a (fixed) toroidal compactification $\mathfrak S(v)^{\rig,\tor}$ and its global sections $\M_{\mc U}$ are by definition the families of ($v$-overconvergent and $w$-locally analytic) Siegel forms of weight $\chi_{\mc U}^{\un}$. There is also a Hecke algebra that acts on $\M_{\mc U}$, including a completely continuous operator $\U_p$.

Unfortunately, it turns out that $\M_{\mc U}$ is not projective in Buzzard's sense, so one can not apply the abstract machinery of \cite{buzz_eigen} to build the eigenvariety. The solution of \cite{AIP} is to consider the space $\M^0_{\mc U} \subset \M_{\mc U}$ of \emph{cuspidal forms}, i.e. sections of $\underline \omega_{v,w}^{\dagger \chi_{\mc U}^{\un}}(-D)$, where $D$ is the boundary of $\mathfrak S(v)^{\rig,\tor} \times \mc U$. They are able to prove that $\M^0_{\mc U}$ is projective and hence they obtain a ($g$-dimensional) eigenvariety. To generalise this result to the non-cuspidal case we need first of all to understand why the module $\M_{\mc U}$ is not projective. To do this it is convenient to work with the minimal compactification $\mathfrak S(v)^{\rig,\ast}$. Let $\pi \colon \mathfrak S(v)^{\rig,\tor} \times \mc U \to \mathfrak S(v)^{\rig,\ast} \times \mc U$ be the natural morphism and let $\mc J^0 \subset \mc O_{\mathfrak S(v)^{\rig,\tor} \times \mc U}$ be the sheaf of ideals corresponding to the boundary of $\mathfrak S(v)^{\rig,\tor} \times \mc U$. By definition we have
\[
\M^0_{\mc U} = \Homol^0(\mathfrak S(v)^{\rig,\tor} \times \mc U, \underline \omega_{v,w}^{\dagger \chi_{\mc U}^{\un}} \otimes \mc J^0).
\]
The boundary of $\mathfrak S(v)^{\rig,\ast}$ is given by Siegel varieties of genus smaller than $g$, and it has in particular a natural stratification given by the union of the varieties of genus smaller than $g-s$, for a given $0 \leq s \leq g-1$. Let $\mc I^s \subset \mc O_{\mathfrak S(v)^{\rig,\ast} \times \mc U}$ be the corresponding sheaf of ideals and let $\mc J^s$ be $\pi^\ast \mc I^s$. Looking at global sections of $\underline \omega_{v,w}^{\dagger \chi_{\mc U}^{\un}} \otimes \mc J^s$ we obtain a filtration
\[
\M^0_{\mc U} \subset \M^1_{\mc U} \subset \cdots \subset \M^g_{\mc U}= \M_{\mc U}.
\]
For example $\M^1_{\mc U}$ is the space of forms that are not necessarily cuspidal but vanish on all the components of the boundary corresponding to Siegel varieties of genus strictly smaller than $g-1$. We define the corank of a given form $f$, denoted $\cork(f)$, by
\[
\cork(f) = \min\{q \mbox{ such that } f \in \M^q_{\mc U} \}.
\]
Let now $\chi=(\chi_i)_{i=1}^g \in \mc W(w)_g$ be a $p$-adic weight. We define the corank of $\chi$, denoted $\cork(\chi)$, by
\[
\cork(\chi)= \max\{s \mbox{ such that } \chi_g = \chi_{g-1} = \cdots = \chi_{g-s + 1} \}.
\]
We obtain in this way the closed subspace $\mc W(w)_g^s \subset \mc W(w)_g$ given by weights of corank at least $s$. The interest of the corank is the following theorem, proved in \cite{WeiVek}.
\begin{teono}[{\cite[Satz~2]{WeiVek}}]
Let $f \neq 0$ be a classical modular forms of integral weight $k$. Then we have
\[
\cork(f) \leq \cork(k).
\]
\end{teono}
For example, this implies that if we have a classical form $f \neq 0$ that is `completely not cuspidal', in the sense that $f \not \in \M^{g-1}_{\mc U}$, then the weight of $f$ must be parallel. Since classical points are dense in the eigenvariety (a fact that follows from the classicality results of \cite{AIP} and \cite{pel}) we see that we can not have a $g$-dimensional eigenvariety for $\M_{\mc U}$ and in particular $\M_{\mc U}$ can not be a projective $A$-module.

Fix now an integer $q > 0$ and let $\mc U = \Spm(A) \subset \mc W(w)_g^q$ be an admissible open. There are no difficulties to define the sheaf $\underline \omega_{v,w}^{\dagger \chi_{\mc U}^{\un}}$, so we obtain the space $\M^q_{\mc U}$ of families of modular forms of corank at most $q$ with weights in $\mc U$. (The case $q=0$ is the case of cuspidal forms and it is done in \cite{AIP}.) Note that $\mc U$ is now $g-q+1$-dimensional so that $\underline \omega_{v,w}^{\dagger \chi_{\mc U}^{\un}}$ parametrizes families of modular forms in $g-q+1$ variables). One of our main results is the following:
\begin{teono}
The $A$-module $\M^q_{\mc U}$ is projective, so we have a $g-q+1$-dimensional eigenvariety for Siegel eigenforms of corank at most $q$.
\end{teono}

The strategy to prove the theorem is to use the Siegel morphism, that we show to be surjective. The expert reader will recognize that the proof of this theorem is heavily inspired by Hida's work, especially \cite{hida_control}, and its generalisations to non-cuspidal setting \cite{skinn_urb,urban}. 

There are various natural morphisms between the eigenvarieties we construct, for weights of different corank, and the natural question that arises is to glue them. We dot not know if this is possible in general, but we are able to glue their reduced part into a single eigenvariety, proving the following theorem.
\begin{teono}
The reduced eigenvarieties $\mathcal{E}_{g,q}^{q,\mr{red}}$, for $q=0,\ldots, g $, glue into a (non-equidimensional) eigenvariety $\mathcal{E}_{g}$ over $\mc W_g$.
\end{teono}
The strategy to prove the theorem is to show that all our (reduced) eigenvarieties are subspaces of a general eigenvariety constructed, via cohomological methods, by Hansen in \cite{Hansen}. (Hansen's eigenvariety is more general but it is not directly related to families of `true forms'.)

We would like now to sketch the strategy of the proof of the projectivity of $\M^q_{\mc U}$. Let $\mathfrak{S}_{g-1}^{\rig,\ast}$ be the minimal compactification of a component of the boundary of $\mathfrak{S}^{\rig,\ast}$ corresponding to a Siegel variety of genus $g-1$. Let $q>1$, let $\mc W_g \to \mc W_{g-1}$ be the morphism that forgets the last component of the weight, and  $\mc V$ the image of $\mc U$ under this morphism. (Note that as $q>1$ we have that $\mc U$ and $\mc V$ are isomorphic. In the paper we also treat the case $q=1$ which is slightly different.) We show that the pullback to $\mathfrak{S}_{g-1}^{\rig,\ast}$ of a family of modular forms with weights in $\mc U \subset \mc W_g^q$ and corank at most $q$ is a modular form of weight in $\mc V$ of corank at most $q-1$, see Proposition~\ref{prop: identification sheaves}. This is one of the key arguments in the paper: it is proved via Fourier--Jacobi expansion and representation theory for the group $\GL_q$. We want to stress that both assumptions, on the weights and on the corank of the forms, are crucial for this result. Taking the pullback of a form we now get the so-called Siegel morphism, and we prove that there is an exact sequence
\[
0 \to \M^0_{\mc U} \to \M^s_{\mc U} \to \bigoplus \M^{s-1}_{\mc V} \to 0
\]
where the direct sum is over all the cusps of genus $g-1$. Since we already know that $\M^0_{\mc U}$ is projective, we conclude by induction.

We think that the surjectivity of the Siegel morphism for families is in itself a very interesting result, especially because the same is not true in the classical complex setting. 
Moreover, as we have already pointed out, this kind of results have been heavily used in several proofs of Main Conjectures. We believe that our result is very likely to be useful to prove instances of non-ordinary Main Conjectures (as stated in \cite{Benois, Pott}), generalising known results in the ordinary setting, see for example \cite{skinn_urb,urban}.


One advantage of this construction is that we know explicitly the dimension of the eigenvarieties we obtain; if one requires the very natural condition that classical points are dense, then the dimension of our varieties is the maximal that one could allow.  Moreover, in the case of the group $\GSp_4/\Q$ and full level, we are able to prove (under a mild hypothesis) a conjecture of Urban (\cite[Conjecture~5.7.3]{urban_eigen}) about the expected dimension of the irreducible components of these non-equidimensional eigenvariety.

The paper is organized as follows. In Section~\ref{sec: analytic} we study the situation over $\C$. Even if, strictly speaking, we do not need the results over the complex numbers, we find it convenient and instructive to analyse the situation. All the basic ideas of the paper (except one cohomological computation) are already visible in this section. We introduce the Shimura varieties we will work with and we prove a theorem which bounds from above the corank of an automorphic form with the corank of its weight, generalising the main results of \cite{WeiVek} to PEL Shimura varieties. We introduce in great generality the Fourier--Jacobi expansion which will allow us to study the Siegel morphism. In Section~\ref{sec: p-adic Section} we develop the theory of $p$-adic modular forms. We introduce the spaces of modular forms we are interested in and the $p$-adic Siegel morphism, showing that it is surjective. This uses the vanishing of cohomology of a small Banach sheaf in the sense of \cite{AIP}, see Proposition~\ref{prop: coho banach 0}, which is an interesting result on its own. In Section~\ref{sec: eigen} we recall Buzzard's machinery and we actually build the eigenvarieties. We finally study the relations between them when varying $g$ and $q$, explain the gluing process, and prove Urban's conjecture for $\GSp_4/\Q$.
\subsection*{Acknowledgments} This work began while GR was a PhD student at Universit\'e Paris 13 and KU Leuven, to which he is very grateful. The main idea of the paper originated from the {\it groupe de travail} on Hida theory and especially the reading of \cite{PilHida}; GR would like to thank all its participants and in particular Jacques Tilouine. RB would like to thank Fabrizio Andreatta, Adrian Iovita, Vincent Pilloni, Benoît Stroh, and Alberto Vezzani for several useful conversations. This work has greatly benefited from an excellent long stay of GR at Columbia University and several discussions with David Hansen, Zheng Liu, and Eric Urban. We would also like to thank Christian Johansson for addressing us to \cite{LanPolo}, and an anonymous referee for pointing out a problem with a previous version of the paper.

\section{Analytic section} \label{sec: analytic}
The aim of this section is the proof of Theorem \ref{Weiss2} which generalises a result of Weissesauer \cite{WeiVek} giving necessary conditions on the weight of an automorphic form for it to be of a given `degree of cuspidality'. This result is at the basis of the philosophy of this paper, which roughly speaking states that non-cuspidal eigenvarieties must be of  smaller dimension than the weight space. The section starts recalling some notation on Shimura varieties of type A and C and the corresponding automorphic forms. We conclude studying the Siegel morphism; in particular, we give a sheaf theoretic version of it (see Proposition \ref{prop: SiegelSheaves}) whose $p$-adic avatar will be the key ingredient for the construction of eigenvarieties.
\subsection{Symplectic and unitary groups} 
\subsubsection*{Symplectic case}\label{IntroSymp}
Let $F_0$ be a totally real number field and $\mc O_{F_0}$ its ring of integers. For an integer $ a \geq 1$ we let $G$ be the algebraic group over $\mc{O}_{F_0}$ whose $A$-points are
\begin{align*}
\GSp_{2a/F_0}(A)=\set{g \in \GL_{2a}(A) \vert \phantom{d}^tg \iota_a g = \nu(g) \iota_a,\; \nu(g)},
\end{align*}
where $\iota_a$ is the $2a \times 2a$ orthogonal matrix \begin{align*}
\iota_a = {\left( \begin{array}{cc}
0 & -w_a \\
w_a & 0
\end{array}\right)},
\end{align*}
being $w_a$ the longest Weyl element ({i.e.} the anti-diagonal matrix of size $a \times a$). 

This can be seen as the space of transformation of a rank $2a$ lattice $\Lambda_a$ over $\mc O_{F_0}$ which preserve, up to a scalar, the symplectic form defined by $\iota_a$. We shall write $V_a$ for the corresponding vector space over $F_0$ and $\set{e_1,\ldots,e_{2a}}$ for the standard symplectic basis. 

We shall call $\nu$ the factor of similitude and we shall denote its kernel by $\SP_{2a}$.

We shall be interested in the maximal parabolic subgroups of $ \GSp_{2a}$. For $0 \leq s \leq a$ let $\Lambda_{a,s}$ be the subspace of $\Lambda_a$ generated by $\set{e_{1},\ldots,e_{s}}$ (if $s=0$, we mean that the set is empty) and $P_{a,s}$ the parabolic of $\GSp_{2a}$ preserving $\Lambda_{a,s}$.
We have that the Levi of $P_{a,s}$ is isomorphic to $\GSp(\set{e_{s+1},\ldots,e_{a-s},e_{a+1},\ldots,e_{2a-s}}) \times \GL(\Lambda_{a,s}) $. Explicitly, we can see this Levi in $ \GSp_{2a}$ as 
\begin{align*}
\left( 
\begin{array}{cccc}
g & 0 & 0 & 0 \\
0 & A & B & 0 \\
0 & C & D & 0 \\
0 & 0 & 0 & \nu(g')w_s^tg^{-1}w_s
\end{array}
\right), \;\;\; g'=\left( 
\begin{array}{cc}
A & B \\
C & D 
\end{array}
\right) \in \GSp_{2a-2s/F_0},\; g \in \GL_{s/F_0}.
\end{align*}
Let $N_{a,s}$ be the unipotent radical of $P_{a,s}$; we shall be interested in its center $Z(N_{a,s})$ which can be explicitly written as 
\begin{align*}
\left( 
\begin{array}{cccc}
1 & 0 & 0 & n \\
0 & 1 & 0 & 0 \\
0 & 0 & 1 & 0 \\
0 & 0 & 0 & 1 
\end{array}
\right), \;\;\;\; w_s n w_s= {^t n}, n \in M_s.
\end{align*}
We have an an action of $\GL_s$ on $Z(N_{a,s})$ induced by conjugation inside $ \GSp_{2a}$: $g.n=gnw_s^tgw_s$.

If we want to specify that we are in the situation considered in this Subsection we will say `in the symplectic case'. It is also called the Hilbert-Siegel case and it corresponds to case (C) of \cite{pel}.

We define another group; let $G$ be the algebraic group over $\Z$ whose $A$-point are
\begin{align*}
G(A)=\set{g \in \GL_{2a}(A \otimes_{\Z} \mc O_{F_0} ) \vert \phantom{d}^tg \iota_a g = \nu(g) \iota_a,\; \nu(g) \in A}.
\end{align*}
It differs from $\mr{Res}_{\mc O_{F_0}/\Z}\GSp_{2a}$ for the condition on the rational multiplier and it is important because it is associated with a Shimura variety. Its maximal parabolic subgroups and their Levi and unipotent are defined as for $\GSp_{2a}$ with the extra condition in the multiplier of the symplectic part of the Levi. By a slight abuse of notation, we shall denote the corresponding objects by the same symbol.
\subsubsection*{Unitary case}
Let $F_0$ be a totally real number field and $F$ a totally imaginary quadratic extension of $F$; let $\mc O_{F_0}$ (resp. $\mc O_{F}$) the ring of integers of $F_0$ (resp. $F$). Take two non-negative integers $b \geq a$.   We define the skew-Hermitian matrix 
\begin{align*}
\iota_{a,b} = {\left( \begin{array}{ccc}
0 & 0 & -w_a \\
0 & \varsigma w_{b-a} & 0 \\
w_a & 0 & 0
\end{array}\right)},
\end{align*}
where $\varsigma$ is a totally imaginary element of $F$.

We consider the unitary group $\mr{GU}({b,a})$ over $\mc{O}_{F_0}$ whose $A$-points are 
\begin{align*}
\mr{GU}({b,a})(A)=\set{g \in \GL_{a+b}(A \otimes_{\mc{O}_{F_0}} \mc O_{F} ) \vert  g^\ast \iota_{b,a} g = \nu(g) \iota_{b,a}, \nu(g) \in A},
\end{align*}
where $g^\ast=c(g^t)$, for $c$ the complex conjugation of $F$ over $F_0$. It is a smooth algebraic group over $\Z_{N}$, for a suitable integer $N$. By a slight abuse of notation, we shall sometime call this group $\GU(b,a)_{F/F_0}$. We shall call $\nu$ the factor of similitude and we shall denote its kernel by $\mr U({b,a})$.

Let $\Omega_{b-a}$ be a $b-a$-dimensional lattice over $\mc O_{F}$ corresponding to the skew-Hermitian matrix $\varsigma\mr{Id}_{b-a}$. Denote a integral basis of it by $ \set{w_i}_{i=1}^{b-a}$. Let $\Xi_a$ and $\Upsilon_a$ be two $\mc O_F$ lattices with basis $\set{x_i}_{i=1}^a$ and $\set{y_i}_{i=1}^a$. We let $\Lambda_{a,b} =\Xi_a \oplus \Omega_{b-a} \oplus \Upsilon_a$ and $V_{b,a}=X_a \oplus W_{b-a} \oplus Y_a$ the corresponding $F$ vector space.
We now classify the parabolic of $\GU({b,a})$. Let $0\leq s \leq a$ and let $\Lambda_{b,a,s}$ be the sub-lattice of $\Lambda_{b,a}$ spanned by $\set{y_i}_{i=a-s+1}^a$. (If $s=0$, we assume this set to be empty.) We denote by $P_{b,a,s}$ the parabolic of $\GU(b,a)$ stabilising $\Lambda_{b,a,s}$. The Levi subgroup of $P_{b,a,s}$ can be identified with $\GU(\set{x_1,\ldots,x_{a-s},w_1,\ldots,w_{b-a},y_1,\ldots, y_{a-s}}) \times \GL(\Lambda_{b,a,s})$. The Levi of $P_{b,a,s}$, seen as a subgroup of $\GU(b,a)$,  can be described as 
\begin{align*}
\left( 
\begin{array}{ccc}
\nu(h)w_s (g^{-1})^\ast w_s  & 0 & 0 \\
0 & h & 0 \\
0 & 0 & g 
\end{array}
\right), \;\;\;\; h \in \GU({b-s,a-s}), g \in \GL_{s/F}.
\end{align*}
Note that $\mathbb{G}_m(F)$ embeds in $\GU({b,a})$ and $\nu(\mathbb{G}_m)=\mr N_{F/F_0}$.
We shall denote by $N_{a,s}$ the unipotent of $P_{b,a,s}$ and by $Z(N_{a,s})$ its center. This center can be explicitly written as 
\begin{align*}
\left( 
\begin{array}{ccc}
1 & 0 & n \\
0 & 1 & 0  \\
0 & 0 & 1 
\end{array}
\right), \;\;\;\; n=w_sn^\ast w_s, n \in M_s.
\end{align*}
We have an an action of $\GL_s$ on $Z(N_{a,s})$ induced by conjugation inside $\GU({a,b})$: $g.n= w_s (g^{-1})^\ast w_sng^{-1}$.

In this case we denote by $G$ the algebraic group over $\Z$
\begin{align*}
G(A)=\set{g \in \GL_{a+b}(A \otimes_{\Z} \mc O_{F} ) \vert  g^\ast \iota_{b,a} g = \nu(g) \iota_{b,a}, \nu(g) \in A}.
\end{align*}
\begin{remark}
We are using a quite uncommon definition of symplectic and unitary groups (using the longest Weyl element rather than the identity matrix) but this makes the Hodge-Tate map equivariant for the action of $\GL_b \times \GL_a$. Moreover the Borel subgroups are always upper triangular.
\end{remark}
If we want to specify that we are in the situation considered in this Subsection we will say `in the unitary case'. It corresponds to case (A) of \cite{pel}.
\subsection{Shimura varieties and their compactification}\label{Shimvar}
Fix $a$ or $a,b$ and let $G$ be as in the previous section. For each $0\leq s \leq a$ we shall write $G_s$ for the corresponding group associated with $\GSp_{2a-2s}$ or $\GU({b-s,a-s})$; according if $G$ is symplectic or unitary. Let $\mc H$ be a compact open subgroup of $G(\mathbb{A}_{\Q,f})$. 
\begin{ass}
We shall assume that $\mc H$ is neat, in the terminology of \cite[1.4.1.8]{lan}.
\end{ass}
Associated with $G$ and $\mc H$ comes a moduli problem for abelian scheme which, under the assumption of neatness, is representable by a quasi projective scheme $S_G(\mc H)$ defined over a number field $K$. We have a minimal (or Baily--Borel) compactification  $S^\ast_G(\mc H)$ and we choose one and for all a smooth toroidal compactification $S^{\tor}_G(\mc H)$ \cite{lan}. We shall denote by $\pi$ the morphism from $S^{\tor}_G(\mc H)$ to $S^\ast_G(\mc H)$.
For $0\leq s \leq a$  we define the set of cusp label of genus $s$ 
\begin{align*}
C_s(\mc H)\colonequals  & (G_{s}(\mathbb{A}_{\Q,f})\times  \GL(\Lambda_{b,a,s}))N_{a,s}(\mathbb{A}_{\Q,f}) \setminus G(\mathbb{A}_{\Q,f}) / \mc H \\
(\mr{resp. }\; C_s(\mc H)\colonequals  & (G_s(\mathbb{A}_{\Q,f}) \times \GL(\Lambda_{a,s}))N_{a,s}(\mathbb{A}_{\Q,f}) \setminus G(\mathbb{A}_{\Q,f}) / \mc H).
\end{align*}
This is a finite set and we shall denote by $[\gamma]$ a generic element of this double quotient. 
We shall write \begin{align*}
\mc H_{[\gamma]} \colonequals  \gamma\mc H \gamma^{-1} \cap G_s(\mathbb{A}_{\Q,f})
\end{align*}
where we see $G_s$ as a component of the Levi of the parabolic $P_{b,a,s}$ (resp. $P_{a,s}$) of $G$. We can then define a stratification of $S^\ast_G(\mc H)$ as follows 
\begin{align*}
S^\ast_G(\mc H)= \bigsqcup_{s=0}^a \bigsqcup_{[\gamma] \in C_s(\mc H)} S_{G_s}(\mc H_{[\gamma]}). 
\end{align*}
If $s=0$, $C_0$ consists of $S_G(\mc H)$. For $s=a$ we obtain compact Shimura varieties.

We want to explicitly compute the stalks of the structural sheaf of $S^\ast_G(\mc H)$. This will be useful to define the Fourier--Jacobi expansion of automorphic forms.

We begin working with the unitary case. We consider the abelian scheme $\mc Z_{[\gamma]}\rightarrow  S_{G_s}(\mc H_{[\gamma]})$ defined in \cite[\S 2.6]{WanLan}. It is isogenous to a certain power of a universal abelian variety for $G_s$.  Let $N_{[\gamma]}\colonequals  \mc H_{[\gamma]} \cap Z(N_{a,s})(\Q)$; this group can be identified with a lattice in the group of $s \times s$ Hermitian matrices with $F$ coefficients. Indeed, for each $n$ we can define a unique Hermitian paring $(y,y')\mapsto b_n(y,y')$ so that $\mr{Tr}_{F/F_0}(b_n(y,y'))= \langle y(n-1),y'\rangle_s$ (where $\langle\phantom{a},\phantom{b}\rangle_s$ is the Hermitian form on $Y_s \times Y_s$ defined before) and consequently a unique $s \times s$ Hermitian matrix. Similarly in the symplectic case.\\
We define $S_{[\gamma]}\colonequals  \Hom_{\Z}(N_{[\gamma]},\Z)$ which, under the above identification, is a lattice in the set of symmetric matrices.  To each element $h \in S_{[\gamma]}$ we can associate $\mc L(h)$, a $\mathbb{G}_m$-torsor on $\mc Z_{[\gamma]}$. We shall denote by $S^+_{[\gamma]}$ the subset of totally non-negative (with respect to the embeddings of $F_0$ in $\mathbb{R}$) elements.\\
Let us denote by $\G_{[\gamma]}\colonequals \GL(Y_s) \cap \gamma\mc H \gamma^{-1}$. This group acts on $H^+_{[\gamma]}$ via the following formula 
\begin{align*}
g.h &=g^{-1}  h w_s(g^{-1})^\ast w_s.
\end{align*}

We now deal with  the symplectic case. We have the abelian scheme $\mc Z_{[\gamma]}$. Let $N_{[\gamma]}\colonequals  \mc H_{[\gamma]} \cap Z(N_{a,s})(\Q)$; this group can be identified with a lattice in the group of symmetric $s \times s$ matrices with $F_0$ coefficients. \\
We define $S_{[\gamma]}\colonequals  \Hom_{\Z}(N_{[\gamma]},\Z)$ which, under the above identification, is a lattice in the set of symmetric matrices.  To each element $h \in S_{[\gamma]}$ we can associate $\mc L(h)$, a $\mathbb{G}_m$-torsor on $\mc Z_{[\gamma]}$. We shall denote by $S^+_{[\gamma]}$ the subset of totally non-negative elements.\\
Let us denote by $\G_{[\gamma]}\colonequals \GL(Y_s) \cap \gamma\mc H \gamma^{-1}$. This group acts on $H^+_{[\gamma]}$ via the following formula 
\begin{align*}
g.h &=w_s^tgw_sh g.
\end{align*}
\begin{prop}\label{stalkmin}
Let $x$ be a closed point of $S_{G_s}(\mc H_{[\gamma]})$. 
 The completion of the strict henselisation  of the stalk is of $\mc O_{S^\ast_{G}(\mc H)}$ at $x$ is canonically isomorphic to 
\begin{align*}
\set{\sum_{h \in S^+_{[\gamma]}} a(h)q^h \:\;\vert \:\; a(h) \in \Homol^0(\hat{\mc Z}_{[\gamma],x},\mc L(h))}^{\G_{[\gamma]}}.
\end{align*}
\end{prop}
\begin{proof}
We shall sketch a proof of this well known fact in the symplectic case, following \cite[Chapter IV]{FaltingsChai}, as this will be useful in the following. 

Fix a cusp label $[\gamma]$. Over $S_{G_s}(\mc H_{[\gamma]})$ we have the universal abelian scheme $\mc A_{[\gamma]}$ and over it the abelian scheme $\mc Z_{[\gamma]}$ which is isogenous to $\mc A_{[\gamma]}^{s}$. We can define a $\mr{Hom}(S_{[\gamma]},\m G_m)$-torsor $M_{[\gamma]}\rightarrow \mc Z_{[\gamma]}$. For every cone $\sigma$ in  $\mc C(\Lambda_a/{(\gamma \Lambda_{a,s})}^{\perp})$ (the cone of positive semi-definite forms on $\left(\Lambda_a/\gamma \Lambda_{a,s}^{\perp}\right) \otimes \mathbb{R}$), we have a toroidal embedding $M_{[\gamma]} \hookrightarrow M_{\sigma}$ (\cite[Definition 6.1.2.3]{lan}). 
Summing up:
$$
    \xymatrix{ M_{[\gamma]} \ar[d] \ar@{^{(}->}[r] \ar@/_2pc/[ddd]_-{g} & M_{\sigma}\ar@/^/[dddl]^-{g'} \\
               \mc Z_{[\gamma]} \ar[d]& \\
               \mc A_{[\gamma]}\ar[d] &\\
               S_{G_s}(\mc H_{[\gamma]}) &}
$$
and hence
\begin{align*}
g_* \mc O_{M_{[\gamma]}} \cong &\oplus_{h \in S_{[\gamma]}} \mr H^0(Z_{[\gamma]}, \mc L(h)),\\
g'_* \mc O_{M_{\sigma}} \cong & \oplus_{h \in S_{[\gamma]} \cap \sigma^{\wedge}} \mr H^0(Z_{[\gamma]}, \mc L(h)),
\end{align*}
for $\sigma^{\wedge}$ the dual cone of $\sigma$. 
One can take the direct limit of all the $M_{\sigma}$, $\sigma \in \mc C_{[\gamma]}$, a smooth and projective admissible polyhedral decomposition of $\mc C(\Lambda_a/{(\gamma \Lambda_{a,s})}^{\perp})$, and obtain a scheme $\overline{M}_{[\gamma],\mc C_{[\gamma]}}$ which contains $M_{[\gamma]}$ as an open dense sub-scheme. 
Moreover, we now have an action of $\Gamma_{[\gamma]}$ which permutes the $h$ in $S_{[\gamma]}$.
We can stratify $\overline{M}_{[\gamma],\mc C_{[\gamma]}}$ by locally closed sub-schemes $Z_{\sigma}$, where each $Z_{\sigma}$ is given by the vanishing locus of the $h$-components, for $h$ positive definite on $\sigma$. So far, everything is algebraic.

We can complete $\overline{M}_{[\gamma],\mc C_{[\gamma]}}$ along the closed stratum $Z_{\sigma}$ (which corresponds to $\gamma \Lambda_{a,s}$) and obtain $\widehat{\overline{M}}_{[\gamma],\mc C_{[\gamma]}}$. We have 
\begin{align}\label{eq:pushfromtor}
(\widehat{\overline{M}}_{[\gamma],\mc C_{[\gamma]}} \rightarrow S_{G_s}(\mc H_{[\gamma]}))_* \mc O_{\widehat{\overline{M}}_{[\gamma],\mc C_{[\gamma]}}} \cong &\prod_{h \in S^+_{[\gamma]}} \mr H^0(Z_{[\gamma]}, \mc L(h)).
\end{align}

We denote by $\hat{S}^{\mr{tor}}_{G_s}(\mc H_{[\gamma]})_{\sigma}$ the quotient of $\widehat{\overline{M}}_{[\gamma],\mc C_{[\gamma]}}$ by $\Gamma_{[\gamma]}$. 
One can find a finite number of geometric points $\overline{x}$ of $\hat{S}^{\mr{tor}}_{G_s}(\mc H_{[\gamma]})_{\sigma}$ such that $\hat{S}^{\mr{tor}}_{G}(\mc H_{[\gamma]})_{\sigma}$ is covered by $\mr{Spec}(R_{\overline{x}})$, for $R_{\overline{x}}$ a finite \'etale extension of the strict local ring of $\overline{x}$.

 The toroidal compactification, associated with our choices of smooth and projective admissible polyhedral decomposition, is constructed gluing the different $\mr{Spec}(R_{\overline{x}})$ (along the faces of the corresponding cones $\sigma$'s). We denote by $S^{\mr{tor}}_{G}(\mc H_{[\gamma]})_{\sigma}$ the stratum in $S^{\mr{tor}}_{G}(\mc H)$ associated with $\sigma$. We have that the completion of $S^{\tor}_G(\mc H)$ along $S^{\mr{tor}}_{G}(\mc H_{[\gamma]})_{\sigma}$ is isomorphic to $\hat{S}^{\mr{tor}}_{G}(\mc H_{[\gamma]})_{\sigma}$. 
We also have that the  
$$ \pi: S^{\mr{tor}}_{G}(\mc H_{[\gamma]})_{\sigma} \rightarrow S_{G_s}(\mc H_{[\gamma]})$$
 is of the form \cite[Theorem 7.2.4.1 (5)]{lan}
\begin{align}\label{eq:tor_su_min_strato}
\overline{M}_{\sigma}/ \Gamma_{[\gamma]}\rightarrow \mc Z_{[\gamma]} \rightarrow S_{G_s}(\mc H_{[\gamma]}).
\end{align}

This and \eqref{eq:pushfromtor} imply the description of the completion  of the strict henselisation of the structure sheaf of the minimal compactification at $x$.

\end{proof}
\begin{remark}
Note that $g \in \G_{[\gamma]}$ sends $\Homol^0(\mc Z_{[\gamma]},\mc L(h))$ into $\Homol^0(\mc Z_{[\gamma]},\mc L(g.h))$. If $\G_{[\gamma]}(h)$ denotes the stabiliser of $h$ in $\G_{[\gamma]}$, the action of $\G_{[\gamma]}(h)$ on $\Homol^0(\mc Z_{[\gamma]},\mc L(h))$ is trivial (see again \cite[Lemma 5.1]{skinn_urb}).
\end{remark}
\begin{remark}
Because of the way the toroidal compactification $S^{\tor}_G(\mc H)$ is constructed, gluing the formal models $\hat{M}_{\sigma}$ and then using Artin's algebrisation theorem, there is no canonical description of the stalks, but only of their henselisation and completion.
\end{remark}

\subsection{\texorpdfstring{Algebraic representations of $\GL_{a+b}$}{Algebraic representations of GL{a+b}}} \label{subsec: alg rep GL}
A the beginning of this subsection we assume for simplicity that $F_0 = \Q$. Consider $\GL_{b} \times \GL_{a}/\Z$ with $b\geq a$; let $B_b$ be the Borel subgroup of upper triangular matrices, $T_b$ the split torus and $N_b$ the unipotent part. Let $B_a^o$ be the Borel subgroup of lower triangular matrices, $T_a$ the split torus and $N_a^o$ the unipotent part. We shall denote by $w_{b,a}$ the matrix $\left(\begin{array}{cc}
0 & w_b \\
\mr{Id}_a & 0
\end{array}\right)$ which represents the permutation 
$$ \left( \begin{array}{ccccccc}
 1 & 2 & \ldots & b &  b+1  & \ldots & b+a \\
 b+a & b+a-1 & \ldots & a+1 & 1  & \ldots &  a 
\end{array} \right).$$ 
 Let $k=(k_1,\ldots,k_b, k_{b+1},\ldots, k_{a+b})$ be a weight of $T_b \times T_a$ and consider the algebraic induction $L_k$ which, for each $\Z$-algebra $R$, is
\begin{align*}
L_k(R)\colonequals \set{f:R[\GL_{b} \times \GL_{a}]\rightarrow R | f(gnt)=k(t)f(g) \forall t \in T_b \times T_a, n \in N_b \times N^o_a},
\end{align*}
where $k(t_1,\ldots,t_b, t_{b+1},\ldots, t_{a+b})=t_{1}^{k_{1}} \cdots t_{b}^{k_{b}} t_{b+1}^{k_{b+1}} \cdots t_{a+b}^{k_{a+b}}$. It is a representation of $\GL_{b} \times \GL_{a}$ via $g.f(g')=f(g^{-1}g')$. We shall sometimes write $\rho_k$ to denote this representation.\\
We say that the weight is dominant (w.r.t. $N_b \times N^o_a$ inside the unitary group) if $k_1 \geq \ldots \geq k_b \geq  k_{b+a}\geq \ldots  \geq k_{b+1}$.

In the symplectic case, we consider $\GL_a$ with the Borel $B_a$ subgroup of upper triangular matrices with its split torus $T_a$ and unipotent radical $N_a$. For any weight $k$ we define the space $$L_k(R)\colonequals \set{f:R[\GL_{a}]\rightarrow R | f(gnt)=k(t)f(g) \forall t \in T_a, N_a}.$$ A weight is said to be dominant (w.r.t. $N_a$) if $k_1 \geq \ldots \geq k_a$.

One can extend all these definition to the case $F_0 \neq \Q$ just considering the Weil restriction functor. In particular we have the representation $L_k$ in general. It is an algebraic representations of $\mr{Res}_{\mc O_{F_0} / \Z}\GL_{b} \times \GL_{a}$ that is dominant w.r.t.  $\mr{Res}_{\mc O_{F_0} / \Z}N_b \times N^o_a$. In this context a weight will be an element of ${\Z[\Sigma]}^{b+a}$, being $\Sigma=\mr{Hom}(F_0,\C)$.  
\subsection{Algebraic automorphic forms}\label{AAF}
Consider $G$ and $\mc H$ as before. We have a universal abelian variety  $\mc A= (\mc A, \lambda, \iota, \eta)$  with PEL structure and a morphism
\begin{align*}\xi: \mc A \rightarrow S_G(\mc H).
\end{align*}  We can extend $\mc A$ to a semi-abelian variety $\mc G$ such that $\xi$ too extends to \begin{align*}\xi: \mc G \rightarrow S^{\tor}_G(\mc H).
\end{align*}
Let us denote by $e$ the unit section of $\xi$ and by $\omega$ the sheaf $e^\ast \Omega_{\mc G/S^{\tor}_G(\mc H)}$. 

In the symplectic case, we have a decomposition 
\begin{align*}
\omega \cong \bigoplus_{\sigma \in \Sigma} \omega_{\sigma} 
\end{align*}
and we define $\mc E $ to be 
\begin{align*}
 \mc E \cong & \bigoplus_{\Sigma} \isom(\mc O^a_{S^{\tor}_G(\mc H)},\omega_{\sigma}).
\end{align*}
This defines an (algebraic) left ${\GL_{a}}$-torsor, where ${\GL_{a}}$ acts on $\mc O^a_{S^{\tor}_G(\mc H)}$ on the right.
The sheaf of weight $k$ automorphic forms is then 
\begin{align*}
\omega^k= {\mc E}^{\mr{Res}_{\mc O_{F_0} / \Z}N_a}[-w_{a}k].
\end{align*}
Locally for the Zariski topology, this sheaf is isomorphic to $L_{-w_{a}k}$; indeed we have also 
\begin{align*}
\omega^k= {\mc E} \times^{\mr{Res}_{\mc O_{F_0} \GL_a}} L_{-w_{a}k}.
\end{align*}
For the unitary case, fix a CM type $(\Sigma,\Sigma^c)$ for $(F_0,F)$; over $F$ we have a decomposition 
\begin{align*}
\omega \cong \bigoplus_{\sigma \in \Sigma} \omega_{\sigma} \oplus \omega_{c\sigma} 
\end{align*}
and we define $\mc E = \mc E^+ \oplus \mc E^-$, where
\begin{align*}
 \mc E^+ \cong & \bigoplus_{\Sigma} \isom(\mc O^a_{S^{\tor}_G(\mc H)},\omega_{\sigma}),\\ 
 \mc E^- \cong & \bigoplus_{\Sigma^c} \isom(\mc O^b_{S^{\tor}_G(\mc H)},\omega_{\sigma}).
\end{align*}
This defines an (algebraic) ${\GL_{b} \times \GL_a}$-torsor.
The sheaf of weight $k$ automorphic forms is then 
\begin{align*}
\omega^k= {\mc E}^{\mr{Res}_{\mc O_{F_0} / \Z}N_b \times N^o_a}[-w_{b,a}k].
\end{align*}
Locally for the Zariski topology, this sheaf is isomorphic to $L_k$ and  as before we have
\begin{align*}
\omega^k= {\mc E} \times^{ \mr{Res}_{\mc O_{F_0} \GL_b \times \GL_a}} L_{-w_{b,a}k}.
\end{align*}
\begin{defin}
For any $\Z$-algebra $R$ we define the space of weight $k$ modular forms as
\begin{align*}
\M_k(\mc H,R) \colonequals \Homol^0({S^\ast_G(\mc H)}_{/R},\pi_\ast\omega^k)=\Homol^0({S^{\tor}_G(\mc H)}_{/R},\omega^k).
\end{align*}
\end{defin}
\begin{remark}
If the boundary of $S^\ast_G(\mc H)$ is of codimension strictly bigger than one, than $\M_k(\mc H,R)=\Homol^0({S_G(\mc H)}_{/R},\omega^k)$.
\end{remark}
We have the following theorem about Fourier--Jacobi expansion \cite[Section 5.3]{LanFJ};
\begin{teo}\label{stalkform}
Let $x$ be a geometric point of $S_{G_s}(\mc H_{[\gamma]})$. The completion of the stalk of $\pi_\ast\omega^k$ is isomorphic to 
\begin{align*}
\set{\sum_{h \in H^+_{[\gamma]}} a(h)q^h \:\;\vert \:\; a(h) \in \Homol^0(\hat{\mc Z}_{[\gamma],x},\mc L(h)\otimes \omega^k )}^{\G_{[\gamma]}}
\end{align*}
where invariance by $\G_{[\gamma]}$ means $a(h)=\rho_k(\gamma^{-1}g\gamma)a(g.h)$,  the action $g.h$ is the one defined before Proposition \ref{stalkmin}. 

If $f$ is an element of the stalk, then the elements $a(h)$ can be take in 
$ \mr H^0(\mc Z_{[\gamma]} \times U,\mc L(h) \omega^k)$, for $U$ containing $x$. 
\end{teo}
\begin{proof}
The proof goes exactly as in the proof of Proposition \ref{stalkmin}, noticing that the have to take the invariant by $\Gamma_{[\gamma]}$ of  
\begin{align}\label{eq:omegasustrato}
\omega^k \rightarrow \overline{M}_{\sigma} \rightarrow \mc Z_{[\gamma]} S_{G_s}(\mc H_{[\gamma]}).
\end{align}

For the second statement, we note that, exactly as in the  section on Fourier--Jacobi expansion in \cite[V.1]{FaltingsChai}, we just need to base change \eqref{eq:omegasustrato} to $U$.
\end{proof}

\subsection{The Siegel morphism} \label{subsec: sieg morph}
In this section we shall study under which condition on the weight $k$ we can have a non-cuspidal form $f$ of that weight. This section is highly influenced by the work \cite{WeiVek}.

Let $\iota_{[\gamma]} : S_{G_s}^\ast(\mc H_{[\gamma]}) \rightarrow S^\ast_G(\mc H)$ be the component of the boundary of the minimal compactification associated with the cusp label $[\gamma]$.
We define the Siegel operator 
\begin{align*}
\Phi_{[\gamma]} (f)= f_{\vert_{S_{G_s}(\mc H_{[\gamma]})} } \in \Homol^0(S^\ast_{G_s}(\mc H_{[\gamma]}), \iota_{[\gamma]}^*\pi_* \omega^k).
\end{align*}
In the symplectic case, for a weight $k= (k_1, \ldots, k_a)$ we define  $k'=(k_1, \ldots, k_{a-s})$, while in the unitary case, for a weight $k= (k_1, \ldots, k_b,k_{b+1},\ldots, k_{a+b})$ we define  $k'=(k_1, \ldots, k_{b-s},k_{b+1},\ldots, k_{a+b-s})$. The following lemma is very important.
\begin{lemma}
Let $R$ be a subfield of $\C$. If $f \in \M_k(\mc H,R)$ then
\[
\Phi_{[\gamma]} (f) \in \M_{k'}(\mc H_{[\gamma]},R).
\]
\end{lemma}
\begin{proof}
Using the comparison of algebraic and analytic Fourier--Jacobi expansion \cite{LanFJ}, we can rephrase the above proof in analytic terms. This has been done, for example, for Siegel forms in \cite[\S 5]{VanderGeer} and for unitary forms in \cite[\S 3.6]{Hsieh}.
\end{proof}
\begin{rmk}\label{rem:FaltingSiegel}
There is an alternative definition of the Siegel morphism $\Phi_{[\gamma]}$, as in \cite[V.1]{FaltingsChai}. Let $x$ be a point of $S_{G_s}(\mc H_{[\gamma]})$, and consider the Fourier--Jacobi expansion at this point. 
Let $f=\sum_h a(h)q^h$ be the image of $f \in \M_k(\mc H,R)$. When restricting to $S_{G_s}^\ast(\mc H_{[\gamma]})$ all the $a(h)$ are sent to $0$ except $a(0)$. But $$a(0) \in \Homol^0({\mc Z}_{[\gamma]},\mc L(0)\otimes \omega^k ) = \Homol^0(S_{G_s}(\mc H_{[\gamma]}), \omega^k ),$$
where we used that $\mc L(0)\otimes \omega^k =\mc O_{{\mc Z}_{[\gamma]}}\otimes \omega^k$ comes via pullback from the open Shimura variety. So $\Phi_{[\gamma]}(f)=a(0)$. The fact that $\Phi_{[\gamma]}(f)$ extends to $S_{G_s}^\ast(\mc H_{[\gamma]})$ is simply because $f$ is already defined on the boundary.
\end{rmk}
We give two key definitions:
\begin{defin}
In the symplectic case, we say that a weight has corank $q$ if
\begin{align*}
q = \vert\set{1\leq i \leq a \vert k_i=k_a}  \vert
\end{align*}
and $k_a$ is parallel.

In the unitary case, we say that a weight $k$ has corank $1 \leq q \leq a$ if $k_b -k_{b+1}$ is parallel and 
\begin{align*}
q = \vert\set{1\leq i \leq b \vert k_i=k_b}  \vert=\vert\set{b+1\leq i \leq a+b \vert k_i=k_{b+1}}  \vert.
\end{align*}
If there is no $q$ for which $k$ satisfies the above conditions, we say that $k$ has corank $0$.
\end{defin}
\begin{defin}
We say  that $0 \neq f \in \M_k(\mc H,R)$ has corank $q$ if $q$ is the minimal integer such that $\bigoplus_{[\gamma] \in C_{q+1}(\mc H)} \Phi_{[\gamma]}(f)=0$. 
(We assume $C_{a+1}$ to be empty.)
\end{defin}
We shall write $\M^q_k(\mc H,R)$ for the subspace of $\M_k(\mc H,R)$ of forms of corank at most $q$. 

We define $\mc J^q$ to be the sheaf of ideals associated with
\[
\bigsqcup_{s=q+1}^a \bigsqcup_{[\gamma] \in C_s(\mc H), \sigma \in \mc C_{[\gamma]}} S^{\mr{tor}}_{G_s}(\mc H_{[\gamma]})_{\sigma}  \hookrightarrow S^{\mr{tor}}_G(\mc H).
\]
We have the following proposition.
\begin{prop}\label{stalkJs}
Let $x$ be a closed $\overline{\Q}$-point of $S_{G_s}(\mc H_{[\gamma]})$. The completion of the stalk of $\pi_\ast(\omega^k \otimes \mc J^q)$  is isomorphic to 
\begin{align*}
 \prod_{[h] \in H^+_{[\gamma]}/\G_{[\gamma]}, \mr{rk}(h)\geq s- q} {\Homol^0(\hat{\mc Z}_{[\gamma],x},\mc L(h)\otimes \omega^k)}^{\G_{[\gamma]}(h)},
\end{align*}
where $\G_{[\gamma]}(h)$ is the subgroup of $\G_{[\gamma]}$ which stabilises $h$.
\end{prop}
\begin{proof}
Indeed, by the construction of the toroidal compactification and its stratification, the maximal ideal of $\mc O_{S^\ast_{G}(\mc H),x}$ is generated by the $q^h$'s and the elements which generates the ideal $\pi_*\mc J_x^q\mc O_{S^\ast_{G}(\mc H),x} $ are exactly the $q$ in the sum above. Hence
\begin{align*}
{\pi_\ast(\omega^k \otimes \mc J^q)}^{\wedge}_x=\set{\sum_{h \in H^+_{[\gamma]}, \mr{rk}(h)\geq s- q} a(h)q^h \:\;\vert \:\; a(h) \in \Homol^0(\hat{\mc Z}_{[\gamma],x},\mc L(h)\otimes \omega^k) }^{\G_{[\gamma]}}.
\end{align*}
\end{proof}
We can now give the main theorem of this section, which is a generalisation of \cite[Satz 2]{WeiVek}:
\begin{teo}\label{Weiss2}
Let $R$ be a subfield of $\C$. If $0 \neq f \in \M_k(\mc H,R)$  then $\mr{cork}(k) \geq \mr{cork}(f)$.
\end{teo}
\begin{proof}
If a form has at least corank $q$ than there exists at at least a cusp label $[\gamma]$ in $C_a(\mc H)$ (so of minimal genus)  such that the Fourier--Jacobi expansion at that cusp has at least a non zero coefficient $a(h)$, for $h$ a matrix of rank $a-q$. In particular, this means that the space of invariants ${L_k(\overline{\Q}) }^{\G_{[\gamma]}(h)}$ is not zero. Let us calculate this space.

We know that we can write, in a suitable basis, $h = \left(\begin{array}{cc}
h' & 0 \\
0 & 0
\end{array}\right) $, where $h'$ is a matrix of size $a-q \times a-q$ and maximal rank. We consider the unitary case now, the symplectic case being similar and easier. It is immediate to see that all matrices in $\GL_a(F)$ of the form $\left( \begin{array}{cc}
\mr{Id}_{a-q} & m \\
0 & g'
\end{array} \right)$, with $g'$ in $\GL_q(F)$ and $m \in M_{a-q,q}(F)$, stabilises $h$.   Let $N_{a,q}$ be the unipotent part of parabolic subgroup of $P_{b,a,q} $;  the $\Q$-points of the Levi of $P_{b,a,q}$ are $(\GL_{q}\times\GL_{b-q} \times\GL_{a-q} \times \GL_{q})(F)$.\\
 Using the theory of higher weights (we are over a characteristic zero field), we have that $L_{-w_0k}^{N_{a,q}}$ decomposes, as $(\GL_{q}\times(\GL_{b-q} \times\GL_{a-q}) \times \GL_{q})(F)$-module, as the irreducible representation
\begin{align*}
L_{(-k_{b},\ldots,-k_{b-q})}\otimes L_{(k_1, \ldots, k_{b-q} , k_{b+q+1},\ldots, k_{a+b})}\otimes L_{(-k_{b+1}, \ldots, -k_{b+q}) }.
\end{align*}
 If we intersect $\GL_{q}(F)\times \mr{Id}_{a+b-2q} \times \GL_{q}(F) $ with $\G_{[\gamma]}$ and we obtain a subgroup of finite order in $ \GL_q(\mc O_{F})$ 
(which, we recall, is embedded as the matrices $$\left( 
\begin{array}{ccc}
w_0 (g^{-1})^\ast w_0  & 0 & 0 \\
0 & 1 & 0 \\
0 & 0 & g 
\end{array}
\right)$$ inside $G$).

This group is not Zariski dense in $ \GL_q$, due to the fact that it contains only matrices whose determinant is a unit in $\mc O_F$, but it is not to far from being it. Indeed, let $(-k_{b},\ldots,-k_{b-q})\times (-k_{b+1}, \ldots, -k_{b+q})$  an algebraic characters whose kernel contains $\GL_q \times \GL_q(\mc O_{F_0})$ (so that the space of invariants is not zero). Firstly, we want the representation to factor through the determinant, hence  $k_b=\ldots=k_{b-q+1}$ and $k_{b+1}=\ldots =k_{b+q}$. Then we are left with the representation of $\GL_1$, explicitly $L_{k_b}\otimes L_{-k_{b+1}}$, hence $k_b-k_{b+1}$ must be a parallel weight in $\Z[\Sigma]$.

These are exactly the condition given by the theorem.

Note that, if not zero, then ${L_{-w_0k}(\overline{\Q}) }^{\G_{[\gamma]}} $ is isomorphic to the representation $L_{(k_1, \ldots, k_{b-q} , k_{b+q+1},\ldots, k_{a+b})}$. 
\end{proof}
\begin{rmk}
Note that over a general basis (for example in characteristic $p$) the theorem is not true (for example, non parallel weights which are parallel modulo $p$ can admit non-cuspidal forms). This was already known to Hida (see \cite[Remark 4.8]{HidaPAFS}).
\end{rmk}
We conclude with the following proposition that gives an algebraic description of $\Phi_{[\gamma]}$. Its proof is clearly inspired by \cite[Proposition 5.7]{skinn_urb} and its version in families will be a key ingredient in the construction of non-cuspidal families.
\begin{prop}\label{prop: SiegelSheaves}
Let $R$ be a subfield of $\C$. For a weight $k$ we let  $k'$ be as in the beginning of the section. Suppose $q=\mr{cork}(k)$, we have the following exact sequence of sheaves on ${S^\ast_G(\mc H)}_{/R}$:
\begin{align*}
0 \rightarrow \pi_\ast(\omega^{k}\otimes \mc J^0) \rightarrow  \pi_\ast(\omega^{k}\otimes \mc J^q) \rightarrow \bigoplus_{C_{1}}\iota_{[\gamma],\ast} \pi_{[\gamma],\ast}(\omega^{k'}\otimes \mc J_{[\gamma]}^{q-1})\rightarrow 0, 
\end{align*}
where $\iota_{[\gamma]}$ is the closed inclusion of $S^\ast_{G'}(\mc H_{[\gamma]})$ into ${S^\ast_G(\mc H)}$ and $\pi_{[\gamma]}$ (resp. $\mc J_{[\gamma]}^{s-1}$) is defined as in \ref{Shimvar} (resp. before \ref{stalkJs}) for $S_{G'}(\mc H_{[\gamma]})$.
\end{prop}
\begin{proof}
By {\it fpqc}-descent, we shall check that the sequence is exact on the completion of the stalks using Proposition~\ref{stalkJs}, see for example \cite[\S 5.3]{LanFJ}. Fix one cusp label $\gamma \in C_1$ and let Im be the image of the restriction to the boundary. Suppose that we know the isomorphism:
\begin{align*}
\iota_{[\gamma]}^\ast\im \cong  \pi_{[\gamma],\ast} \omega^{k'} \otimes_{\mc{O}_{S^\ast_{G_s}(\mc H_{[\gamma]})}} \mc J_{[\gamma]}^{q-1},
\end{align*}
then it is immediate to see that $\iota_{[\gamma],\ast} \iota_{[\gamma]}^\ast \im \cong \im $ as $\mathrm{Supp}(\im) \subset \sqcup_{_{C_{1}}} S^\ast_{G_1}(\mc H_{[\gamma]})$.
Let $x$ be a point in a cusp label $[\gamma_1] \in C_{s}(\mc H)$, and let $[\gamma_2] \in C_{s-1}(\mc H_{[\gamma]})$ be the only genus $s-1$ cusp label to which $x$ belongs. 
We have
\begin{align*}
\widehat{\pi_\ast (\omega_k \otimes \mc{J}^0)}_x = & \prod_{[h] \in H^+_{[\gamma_1]}/\G_{[\gamma_1]}, \mr{rk}(h)\geq s} {\Homol^0(\hat{\mc Z}_{[\gamma_1],x},\mc L(h)\otimes \omega^k)}^{\G_{[\gamma_1]}(h)}, \\
\widehat{\pi_\ast (\omega_k \otimes \mc{J}^q)}_x = &  \prod_{[h] \in H^+_{[\gamma_1]}/\G_{[\gamma_1]}, \mr{rk}(h)\geq s-q} {\Homol^0(\hat{\mc Z}_{[\gamma_1],x},\mc L(h)\otimes \omega^k)}^{\G_{[\gamma_1]}(h)}.
\end{align*}

The image is hence 
\begin{align*}
\prod_{[h] \in H^+_{[\gamma_1]}/\G_{[\gamma_1]}, s-q \leq r(h)<s} {\Homol^0(\hat{\mc Z}_{[\gamma_1],x},\mc L(h)\otimes \omega^k)}^{\G_{[\gamma_1]}(h)}.
\end{align*}

By construction of the toroidal compactification, the injection $H^+_{[\gamma_2]} \hookrightarrow H^+_{[\gamma_1]}$ induces an equivalence between $H^+_{[\gamma_2]}/\G_{[\gamma_2]}$ and the matrices of rank smaller than $
s$ of $H^+_{[\gamma_1]}/\G_{[\gamma_1]}$ of the form  $h=\left(\begin{array}{cc}
h' & 0\\
0 & 0
\end{array}\right)$ (in a properly chosen basis, depending only on the cusp label $[\gamma]$).  The description of $\pi_{[\gamma],\ast} (\omega_{k'} \otimes_{\mc{O}_{S^\ast_{G_s}(\mc H_{[\gamma]})}} \mc J_{[\gamma]}^{q-1})$ given by Proposition~\ref{stalkJs} tell us that the completed stalk ${\pi_{[\gamma],\ast} (\omega_{k'} \otimes_{\mc{O}_{S^\ast_{G_s}(\mc H_{[\gamma]})}} \mc J_{[\gamma]}^{q-1})}_x$ is
\begin{align*}
 \prod_{[h'] \in H^+_{[\gamma_2]}/\G_{[\gamma_2]}, s-1 \geq r(h')\geq s-1-(q-1)} {\Homol^0(\hat{\mc Z}_{[\gamma_2],x},\mc L(h')\otimes \omega^{k'})}^{\G_{[\gamma_2]}(h')}.
\end{align*}

We deal only with symplectic case, being the unitary case similar. Unfolding the definitions (in particular, remember that we apply the longest element of the Weyl group to our representation) we see that the action of $\G_{[\gamma_1]}(h)$ factors through 
\begin{align*}
 \set{\left( 
\begin{array}{cccc}
g & 0 & 0 & 0 \\
0 & 1_{g-s} & 0 & 0 \\
0 & 0 & 1_{g-s} & 0 \\
0 & 0 & 0 & \nu(g')w_s^tg^{-1}w_s
\end{array}
\right) \vert g=\left(\begin{array}{cc}
g'' & 0\\
n'' & 1_{s-r(h)}
\end{array}\right) },
\end{align*} 
and similarly for $\G_{[\gamma_2]}(h') $.

We can identify $\mc Z_{[\gamma_1]}$ with $\mr{Hom}(\Lambda_{a,s},\mc A_{[\gamma_1]})$ for $ \mc A_{[\gamma_1]}$ the universal abelian variety over $S_{G_s}(\mc H_{[\gamma_1]})$. Similarly $\mc Z_{[\gamma_2]} \cong \mr{Hom}(\Lambda_{a-1,s-1},\mc A_{[\gamma_2]})$. 

Note that $ \mc A_{[\gamma_1]} \cong  \mc A_{[\gamma_2]}$ and let $\zeta: \mc Z_{[\gamma_1]} \rightarrow \mc Z_{[\gamma_2]}$ be the natural projection. Using the fact that $\omega_k$ is defined only in terms of $\mc A_{[\gamma_1]}$ and not of $\mc Z_{[\gamma_1]}$ we see, as in the proof of  \cite[Lemma 5.1]{skinn_urb}, that $\mc L(h)\otimes \omega^{k}$ comes via pull-back from a sheaf $ \tilde{ \mc L}(h)\otimes \omega^{k}$ on the quotient $\mc Z_{h}$ of $\mc Z_{[\gamma_1]}$ on which $\G_{[\gamma_1]}(h)$ does not act. Hence we have, using the projection formula,
\begin{align*}
{\Homol^0(\hat{\mc Z}_{[\gamma_1],x},\mc L(h)\otimes \omega^{k})}^{\G_{[\gamma_1]}(h)} = & {\Homol^0(\hat{\mc Z}_{h,x}, \tilde{ \mc L}(h)\otimes \omega^{k})}^{\G_{[\gamma_1]}(h)} \\ 
= & {\Homol^0\left(\hat{\mc Z}_{h,x}, {(\tilde{ \mc L}(h)\otimes \omega^{k})}^{\G_{[\gamma_1]}(h)}\right)}.
\end{align*}
We do the same with $\mc Z_{[\gamma_2]}$ and note that by construction $\mc Z_{h} =\mc Z_{h'}$. 
As $\tilde{ \mc L}(h)\otimes \omega^{k}$ is \'etale locally isomorphic to $L_{-w_0k}$, we can explicit ${(\tilde{ \mc L}(h)\otimes \omega^{k})}^{\G_{[\gamma_1]}}$ with the same calculation performed in the proof of Theorem~\ref{Weiss2}. The same calculation for ${(\tilde{ \mc L}(h')\otimes \omega^{k'})}^{\G_{[\gamma_2]}(h')}$ gives 
\begin{align*}
{(\tilde{ \mc L}(h)\otimes \omega^{k})}^{\G_{[\gamma_1]}(h)} \cong {(\tilde{ \mc L}(h')\otimes \omega^{k'})}^{\G_{[\gamma_2]}(h')}
\end{align*}
and this allows us to conclude.
\end{proof}
\begin{remark}
Note again that over a general basis the proposition is not necessarily true.
\end{remark}
In particular, note that we have
 \begin{align*}
\M^q_k(\mc H,R)=\Homol^0({S^{\mr{tor}}_G(\mc H)}_{/R},\omega^k\otimes \mc J^q).
\end{align*}

We conclude observing that all the results of the section can be generalised to arbitrary Shimura varieties of type $A$ or $C$, with the only inconvenience of a less explicit description of the conditions on the weights and even more cumbersome notation.

\section{\texorpdfstring{$p$-adic section}{p-adic Section}} \label{sec: p-adic Section}
Let $p > 2$ be a rational prime number, fixed from now on. We now move on to $p$-adic modular forms. Let $\overline \Q$ be the algebraic closure of $\Q$ in $\C$ and let $\overline \Q_p$ be an algebraic closure of $\Q_p$. We denote with $\C_p$ the completion of $\overline \Q_p$. We fix once and for all an embedding $\overline \Q \hookrightarrow \overline \Q_p$. We assume that $p$ is unramified in $F_0$ and that the ordinary locus of our Shimura variety it is not empty (it is then automatically dense and this is always true in the symplectic case). In the unitary case we moreover assume that each prime above $p$ in $\mc O_{F_0}$ splits completely in $\mc O_{F}$. We are hence in a situation considered in \cite{pel}.

We let $K$ denote a finite extension of $\Q_p$, that we assume to be `sufficiently big' (for example it must contain the image of all the embeddings $F \hookrightarrow \C_p$). This is in contrast with the previous notation for $K$ (it was a number field of definition for the Shimura varieties), but it should not cause any confusion. All our objects will be defined over $K$ or over $\mc O_K$, even if the notation does not suggest it. We assume that the compact open subgroup $\mc H \subset G(\mathbb A_{F_0,f})$ is of the form $\mc H = \mc H^p G(\Z_p)$, where $\mc H^p \subset G(\hat {\mathbb Z}^p)$ is a (sufficiently small) compact open subgroup. In this way $S_G(\mc H)$ and his compactifications have a natural model, denoted with the same symbol, over $\mc O_K$.

Let $p = \prod_{i=1}^k \varpi_i$ be the decomposition of $p$ in $\mc O_{F_0}$ and let $\mc O_i$ be the completion of $\mc O_{F_0}$ with respect to $(\varpi_i)$. (here $\varpi_i$ is a fixed uniformiser of $\mc O_i$.) We have $\mc O_{F_0,p} \cong \prod_{i=1}^k \mc O_i$. We set $d_i \colonequals [F_i:\Q_p]$, where $F_i\colonequals \Fr(\mc O_i)$. From now on, we assume that $K$ is big enough to contain the image of all embeddings $F \hookrightarrow \C_p$. In this section $A$ will be an abelian scheme given by the moduli problem associated to $Y$. We assume that $A$ is defined over a finite extension of $\mc O_K$, so it comes from a rigid point of $\mathfrak Y^{\rig}$.

Let $\mathfrak S^{\tor}_G(\mc H)$ be the formal completion of $S^{\tor}_G(\mc H)$ along its special fiber and let $\mathfrak S^{\tor,\rig}_G(\mc H)$ be the rigid fiber of $\mathfrak S^{\tor}_G(\mc H)$. As in \cite{pel} we have the Hodge height function
\begin{gather*}
\Hdg \colon \mathfrak S^{\tor,\rig}_G(\mc H) \to [0,1]^k \\
x \mapsto (\Hdg(x)_i)_i
\end{gather*}
If $\underline v = (v_i)_i \in [0,1]^k$ we set
\[
\mathfrak S^{\tor,\rig}_G(\mc H)(\underline v) \colonequals \set{x \in \mathfrak S^{\tor,\rig}_G(\mc H) \mbox{ such that } \Hdg(x)_i \leq v_i \mbox{ for all } i} .
\]
We assume that each $v_i$ is small enough in the sequel, as in \cite[Section~1]{pel}. In particular we have the tower of formal schemes
\[
\mathfrak S^{\tor}_G(\mc H p^n)(\underline v) \to \mathfrak S^{\tor}_G(\mc H p^n)_{\Iw}(\underline v) \to \mathfrak S^{\tor}_G(\mc H p)_{\Iw}(\underline v)  \to \mathfrak S^{\tor}_G(\mc H)(\underline v).
\]
We adapt all the notation of \cite{pel}, everything should be clear from the context. For example, we will work with the weight space $\mc W_a$ in the symplectic case and $\mc W_{b,a}$ in the unitary case. Recall that it is the rigid analytic space associated to the completed group algebra $\mc O_K \llbracket \T_a(\Z_p) \rrbracket$ or $\mc O_K \llbracket\T_b \times \T_a(\Z_p) \rrbracket$. As in \cite[Section~2]{pel}, we have, for any tuple of non-negative rational numbers $\underline w$, the affinoid subdomain $\mc W_a(\underline w) \subset \mc W_a$ (and $\mc W_{b,a}(\underline w) \subset \mc W_{b,a}$). Here $\underline w = (w^\pm)_{i=1}^k$ in the unitary case and $\underline w = (w)_{i=1}^k$ in the symplectic case. Let $\mc V \subset \mc W_a$ or $\mc V \subset \mc W_{b,a}$ be an affinoid and let $\chi^{\un}_{\mc V}$ the associated universal character. We will denote the usual involution on the weight space, defined using the longest element of the Weil group, by $\chi \mapsto -w_0 \chi$

Let $\underline w$ be a tuple of rational numbers such that $\chi^{\un}$ is $\underline w$-analytic, so $\mc V \subset \mc W_a(\underline w)$ or $\mc V \subset \mc W_{b,a}(\underline w)$. Let $\underline v$ be adapted to $\underline w$. One of the main construction of \cite{pel} is the sheaf $\underline \omega_{\underline v,\underline w}^{\dagger\chi_{\mc V}^{\un}}$ over $\mathfrak S^{\tor,\rig}_G(\mc H p^n)_{\Iw}(\underline v) \times \mc V$ whose global sections are by definition the families of ($\underline v$-overconvergent and $\underline w$-analytic) modular forms parametrised by $\mc V$, of Iwahoric level. We consider also the rigid spaces $\mc W_a(\underline w)^\circ$ and $\mc W_{b,a}(\underline w)^\circ$ and their formal models $\mathfrak W_a(\underline w)^\circ$ and $\mathfrak W_{b,a}(\underline w)^\circ$ introduced in \cite[Section~5.2]{pel}. It is the correct weight space to consider when working with modular forms of level $\mc H p^n$. If $\mathfrak U \subset \mathfrak W_a(\underline w)^\circ$ or $\mathfrak U \subset \mathfrak W_{b,a}(\underline w)^\circ$ is an affine, we have the sheaf $\underline {\mathfrak w}_{\underline v,\underline w}^{\dagger\chi_{\mathfrak U}^{\un}}$ over $\mathfrak S^{\tor}_G(\mc H p^n)(\underline v) \times \mathfrak U$, where $n$ depends on $\mathfrak U$, and also its rigid fiber $\underline \omega_{\underline v,\underline w}^{\dagger\chi_{\mc U}^{\un}}$. The global sections of $\underline \omega_{\underline v,\underline w}^{\dagger\chi_{\mc U}^{\un}}$ are by definition the families of ($\underline v$-overconvergent and $\underline w$-analytic) modular forms parametrised by $\mc U$, of level $\mc H p^n$. Let $\mc U$ be the rigid fiber of $\mathfrak U$. If $\mc U$ is the image of a given affinoid $\mc V \subset \mc W_a(w)$ under the natural morphism $\mc W_a(w) \to \mc W_a(w)^\circ$ (and similarly in the unitary case), we can recover the sheaf $\underline \omega_{\underline v,\underline w}^{\dagger\chi_{\mc V}^{\un}}$ from $\underline \omega_{\underline v,\underline w}^{\dagger\chi_{\mc U}^{\un}}$. For technical reasons, we will start working with $\underline {\mathfrak w}_{\underline v,\underline w}^{\dagger\chi_{\mathfrak U}^{\un}}$ and $\underline \omega_{\underline v,\underline w}^{\dagger\chi_{\mc U}^{\un}}$

We are going to describe the stalks of the projection to the minimal compactification of $\underline \omega_{\underline v,\underline w}^{\dagger\chi}$.
\subsection{Analytic induction and Fourier expansion} \label{subsec: an ind}
We first of all need to rewrite Subsection~\ref{subsec: alg rep GL} in the $p$-adic setting. Thanks to our assumption that $p$ is unramified in $\mc O_{F_0}$ we can be completely explicit.

We consider the algebraic group $\GL^{\mc O}$ over $\Z_p$ defined, in the unitary and symplectic case respectively, by
\[
\GL^{\mc O} \colonequals \prod_{i=1}^k \Res_{\mc O_i/\Z_p}(\GL_{b} \times \GL_{a}) \mbox{ and } \GL^{\mc O} \colonequals \prod_{i=1}^k \Res_{\mc O_i/\Z_p}\GL_{a}.
\]
We also have the subgroup $\T^{\mc O}$ defined by
\[
\T^{\mc O} \colonequals \prod_{i=1}^k \Res_{\mc O_i/\Z_p}(\m G^{b}_{\operatorname{m}} \times \m G^{a}_{\operatorname{m}}) \mbox{ and } \T^{\mc O} \colonequals \prod_{i=1}^k \Res_{\mc O_i/\Z_p}\m G^{a}_{\operatorname{m}}.
\]
Over $K$, we have that $\T^{\mc O}$ is a split maximal torus of $\GL^{\mc O}$. We consider the Borel subgroup $\B^{\mc O}$ given, in the unitary case, by couples of matrices whose first component is upper triangular and whose second component is lower triangular (in the symplectic case we consider upper triangular matrices). We will write $\U^{\mc O}$ for the unipotent radical of $\B^{\mc O}$. We write $\B^{\mc O, \op}$ and $\U^{\mc O, \op}$ for the opposite subgroups of $\B^{\mc O, \op}$ and $\U^{\mc O, \op}$. Let $\I^{\mc O}$ be the Iwahori subgroup of $\GL^{\mc O}(\Z_p)$ given, in the unitary case, by couples of matrices whose first component has upper triangular reduction and whose second component has lower triangular reduction (in the symplectic case we consider matrices with upper triangular reduction). Let $\N^{\mc O, \op}$ be the subgroup of $\U^{\mc O, \op}(\Z_p)$ given by those matrices that reduce to the identity modulo $p$. We have an isomorphism of groups
\[
\N^{\mc O, \op} \times \B^{\mc O}(\Z_p) \to \I^{\mc O}
\]
given by the Iwahori decomposition.

We use the following identification, in the unitary and symplectic case respectively.
\begin{gather*}
\N^{\mc O, \op} = \prod_{i=1}^k \left( p\mc O_i^{\frac{b(b-1)}{2}} \times p\mc O_i^{\frac{a(a-1)}{2}} \right) \subset \prod_{i=1}^k \left( \mathbb A^{\frac{b(b-1)}{2}, \rig} \times \mathbb A^{\frac{a(a-1)}{2}, \rig} \right),\\
\N^{\mc O, \op} = \prod_{i=1}^k p\mc O_i^{\frac{a(a-1)}{2}} \subset \prod_{i=1}^k \mathbb A^{\frac{a(a-1)}{2}, \rig}.
\end{gather*}
Given $\underline w$, a tuple of positive real numbers as in the definition of the weight spaces, we define in the unitary and symplectic case respectively
\begin{gather*}
\N^{\mc O, \op}_{\underline w} \colonequals \bigcup_{(x_i^\pm) \in \N^{\mc O, \op}} \prod_{i=1}^k \left( B(x_i^+, p^{-w_i^+}) \times B(x_i^-, p^{-w_i^-}) \right),\\
\N^{\mc O, \op}_{\underline w} \colonequals \bigcup_{(x_i) \in \N^{\mc O, \op}} \prod_{i=1}^k B(x_i, p^{-w_i}),
\end{gather*}
where $B(x,p^{-w})$ is the ball of center $x$ and radius $p^{-w}$ inside the relevant affine rigid space.

We say that a function $f \colon \N^{\mc O, \op} \to K$ is \emph{$\underline w$}-analytic if it is the restriction of an analytic function $f \colon \N^{\mc O, \op}_{\underline w} \to K$. We write $\mc F^{\underline w-\an}(\N^{\mc O, \op},K)$ for the set of $\underline w$-analytic functions. If $w_i^\pm = 1$ for all $i$ and $f$ is $\underline w$-analytic, we simply say that $f$ is analytic and we write $\mc F^{\an}(\N^{\mc O, \op},K)$ for the set of analytic functions. A function is \emph{locally analytic} if it is $\underline w$-analytic for some $\underline w$ and we write $\mc F^{\locan}(\N^{\mc O, \op},K)$ for the set of locally analytic functions.

Let now $\chi$ be a $\underline w$-analytic character in $\mc W_{b,a}(\underline w)^\circ(K)$ or $\mc W_a(\underline w)^\circ(K)$. We set
\begin{gather*}
L_\chi^{\underline w-\an, \circ} \colonequals \{f \colon \I^{\mc O} \to K \mbox{ such that } f(it)=\chi(t)f(i)\\
\mbox{ for all } (i,t) \in \I^{\mc O} \times \T_{\underline w}^{\mc O} \mbox{ and } f_{|\N^{\mc O, \op}_{\underline w}} \in \mc F^{\underline w-\an}(\N^{\mc O, \op},K)\},
\end{gather*}
where $\T_{\underline w}^{\mc O}$ is the torus given by (with the obvious meaning of $R/p^{\underline w}R$)
\[
\T_{\underline w}^{\mc O}(R) = \ker (\T^{\mc O}(R) \to \T^{\mc O}(R/p^{\underline w})).
\]
The definition of the spaces $L_\chi^{\an, \circ}$ and $L_\chi^{\locan, \circ}$ is similar. They all are representations of $\I^{\mc O}$ via $(i \star f)(x) = f(xi)$. If $\chi$ is a $\underline w$-analytic character in $\mc W_{b,a}(\underline w)(K)$ or $\mc W_a(\underline w)(K)$ we have the spaces $L_\chi^{\underline w-\an}$, $L_\chi^{\an}$, and $L_\chi^{\locan}$ defined using the action of the whole $\T^{\mc O}$ (or, that is the same, the action of $\B^{\mc O}$).

Let now $\mc U = \Spm(A)$ be an open affinoid in $\mc W_{b,a}(\underline w)^\circ$ or $\mc W_a(\underline w)^\circ$, with universal character $\chi_{\mc U}^{\un}$. We define
\begin{gather*}
L_{\chi_{\mc U}^{\un}}^{\underline w-\an, \circ} \colonequals \{f \colon \I^{\mc O} \to A \mbox{ such that } f(it)=\chi_{\mc U}^{\un}(t)f(i)\\
\mbox{ for all } (i,t) \in \I^{\mc O} \times \T_{\underline w}^{\mc O} \mbox{ and } f_{|\N^{\mc O, \op}_{\underline w}} \in \mc F^{\underline w-\an}(\N^{\mc O, \op},A)\},
\end{gather*}
with the obvious meaning of $\mc F^{\underline w-\an}(\N^{\mc O, \op},A)$. All the space defined above have a relative version over $\mc U$, and we will use the corresponding notation.
\begin{notation}
All our (algebraic) groups have been defined starting with lattices of rank $a$ or $a$ and $b$. We can generalise this definition to other ranks. If we want to stress the ranks we will add certain index. For example $\GL^{\mc O, b-1,a-1}$. These will be the relevant groups when we will consider the cusps of the minimal compactification of our Shimura variety.
\end{notation}
We have the following
\begin{prop} \label{prop: omega circ repr gen}
Let $\chi$ in $\mc W_{b,a}(\underline w)^\circ(K)$ or in $\mc W_a(\underline w)^\circ(K)$ be a $\underline w$-analytic character. Locally for the étale topology on $\mathfrak S_G^{\tor,\rig}(\mc H p^n)(\underline v)$ the sheaf $\underline \omega^{\dagger\chi}_{\underline v,\underline w}$ is isomorphic to $L_{-w_0\chi}^{\underline w-\an, \circ}$. This isomorphism respects the action of $\I^{\mc O}$. An analogous result holds if $\chi$ is in $\mc W_{b,a}(\underline w)(K)$ or $\mc W_a(\underline w)(K)$, considering the étale topology on $\mathfrak S_G^{\tor,\rig}(\mc Hp)_{\Iw}(\underline v)$. Similarly, if $\mc U$ is an open affinoid in $\mc W_{b,a}(\underline w)^\circ$ or $\mc W_a(\underline w)^\circ$, with universal character $\chi_{\mc U}^{\un}$, the sheaf $\underline \omega^{\dagger\chi_{\mc U}^{\un}}_{\underline v,\underline w}$ is étale locally isomorphic to $L_{-w_0\chi_{\mc U}^{\un}}^{\underline w-\an, \circ}$. Analogously for $\mc W_{b,a}(\underline w)$ or $\mc W_a(\underline w)$.
\end{prop}
The Hasse invariants descend to the minimal compactification, allowing us to define the space $\mathfrak S^{\ast}_G(\mc H)(\underline v)$. For any $s=0, \ldots, a$ and for any cusp label $[\gamma] \in C_s(\mc H)$ we have the cusp $\mathfrak S_{G_s}(\mc H_{[\gamma]})(\underline v)$ of $\mathfrak S^{\ast}_G(\mc H)(\underline v)$.

After inverting $p$ there is no problem in adapting the definition of $C_s(\mc H)$, $H^+_{[\gamma]}$, $\Gamma_{[\gamma]}$ etc to the case of level $\mc Hp^n$. Moreover we also have the space $\mathfrak S_{G}^{\ast,\rig}(\mc Hp^n)(\underline v)$, defined using the analytification of the Shimura variety over $K$, and a morphism $\xi \colon \mathfrak S_{G}^{\tor,\rig}(\mc Hp^n)(\underline v) \to \mathfrak S_{G}^{\tor,\rig}(\mc H)(\underline v)$ which is a finite map. We write
\[
\eta \colon \mathfrak S_{G}^{\tor,\rig}(\mc Hp^n)(\underline v) \to \mathfrak S_{G}^{\ast,\rig}(\mc H)(\underline v)
\]
for the composition of $\xi$ with the map to the minimal compactification. We want to describe the stalks of the sheaf $\eta_\ast \underline \omega^{\dagger\chi}_{\underline v,\underline w}$. 

Recall that Huber \cite[(1.1.11)]{huber} defined a fully faithful functor from the category of rigid varieties to adic spaces; we write $\mathfrak S_{G}^{\ast,\mr{ad}}(\mc Hp^n)(\underline v)$, $\mathfrak S_{G}^{\tor,\mr{ad}}(\mc Hp^n)(\underline v)$, etc. for the corresponding adic spaces. This morphism induces an isomorphism of topoi, and we shall see $\eta_\ast \underline \omega^{\dagger\chi_{\mc U}^{\un}}_{\underline v,\underline w}$ as a sheaf on the rigid site or the adic site according to the situation. Note that one can take fiber products between spaces in the image of this functor. 
\begin{prop}\label{prop: stalk min circ}
Let $\mc U$ be an open affinoid in $\mc W_{b,a}(\underline w)^\circ$ or $\mc W_a(\underline w)^\circ$, with universal character $\chi_{\mc U}^{\un}$. Let $[\gamma] \in C_s(\mc H)$ be a cusp label and let $x=\mr{Spa}(K,K^+)$ a geometric adic point of $\mathfrak S_{G_s}^{\mr{ad}}(\mc H_{[\gamma]})(\underline v)$. The completion of the stalk of $\eta_\ast \underline \omega^{\dagger\chi_{\mc U}^{\un}}_{\underline v,\underline w}$ is canonically isomorphic to 
\begin{align*}
\prod_{[\gamma']}\left(\prod_{[h] \in H^+_{[\gamma']}/\G_{[\gamma']}, \mr{rk}(h)\geq s} {\Homol^0(\hat{\mathfrak Z}^{\mr{ad}}_{[\gamma'],x},\mc L(h)\otimes\underline \omega^{\dagger\chi_{\mc U}^{\un}}_{\underline v,\underline w})}^{\G_{[\gamma']}(h)}\right),
\end{align*}
where $\G_{[\gamma']}(h)$ is the subgroup of $\G_{[\gamma']}$ which stabilises $h$, $[\gamma']$ are the cusp labels for $\mathfrak S_{G}^{\tor,\rig}(\mc Hp^n)(\underline v) $ which  map to $[\gamma]$ under $\xi$, and $\hat{\mathfrak Z}^{\mr{ad}}_{[\gamma'],x}$ is the  adic space associated with $\hat{\mc Z}^{\mr{rig}}_{[\gamma'],x}$ (defined through a formal model such as the one in \cite[\S 8.2.1]{AIP}). An analogous result holds if $\mc U$ is in $\mc W_{b,a}(\underline w)$ or $\mc W_a(\underline w)$.

If moreover a section $f$ in the stalk of $\eta_\ast \underline \omega^{\dagger\chi_{\mc U}^{\un}}_{\underline v,\underline w}$ is already defined over an open $\mathfrak{U}$, we have that the coefficients belong to 
$$ \Homol^0(\hat{\mathfrak Z}^{\mr{ad}}_{[\gamma']}\times \mathfrak{U},\mc L(h)\otimes\underline \omega^{\dagger\chi_{\mc U}^{\un}}_{\underline v,\underline w}).$$

In particular, this proves the $q$-expansion principle for (families of) overconvergent modular forms: $f$ is $0$ if and only if its $q$-expansion (on each connected component) is $0$.
\end{prop}
\begin{proof}
First of all, note that the local description of the toroidal and minimal compactification given in Proposition \ref{stalkmin} once one defines $\mathfrak{M}_{[\gamma]}(\underline v)$, $\mathfrak{M}_{\sigma}(\underline v)$ taking the adic fiber of the formal models defined as in \cite[Proposition 8.2.1.2]{AIP}. 

One has a stratification of the boundary also in the adic setting, as Huber's functor is an equivalence of categories if restricted to adic space locally of finite type over a field.

Note that the sheaf $\underline \omega^{\dagger\chi_{\mc U}^{\un}}_{\underline v,\underline w}$ is defined using a torsor $\tilde{\mathfrak{IW}} \rightarrow  \mathfrak S_{G}^{\tor,\rig}(\mc Hp^n)(\underline v)$ which is locally trivial (see \cite[\S 4.5]{AIP} and \cite[\S 2.3]{pel}). As the Iwahori is a union of open balls, Huber adification functor  behaves well on $\tilde{\mathfrak{IW}}$ too and we get 
\begin{align}\label{eq:tor_su_min_stratoadic }
\left(\tilde{\mathfrak{IW}} \times_{\mathfrak S_{G}^{\tor,\rig}(\mc Hp^n)(\underline v)} \overline{ \mathfrak{M}}^{\mr{ad}}_{\sigma}(\underline v) \right)/ \Gamma_{[\gamma]}\rightarrow \mathfrak{Z}^{\mr{ad}}_{[\gamma]} \rightarrow \mathfrak S_{G_s}^{\mr{ad}}(\mc H_{[\gamma]})(\underline v);
\end{align}
similarly on the completion $\widehat{\overline{\mathfrak{M}}}_{[\gamma],\mc C_{[\gamma]}}$.

Once we specialise (the power of) the universal abelian variety at the point corresponding to $x$ we get 
\[ 
\Homol^0(\hat{\mathfrak Z}^{\mr{rig}}_{[\gamma'],x},\mc L(h)\otimes\underline \omega^{\dagger\chi_{\mc U}^{\un}}_{\underline v,\underline w}) \cong  \Homol^0(\hat{\mathfrak Z}^{\mr{rig}}_{[\gamma'],x},\mc L(h))\hat{\otimes} L^{\underline w-\an, \circ}_{-w_0 \chi_{\mc U}^{\un}}.
\] 
(Note that the isomorphism is not canonical but it depends on $x$.)

Taking into account that the morphism $\xi \colon \mathfrak S_{G}^{\tor,\rig}(\mc Hp^n)(\underline v) \to \mathfrak S_{G}^{\tor,\rig}(\mc H)(\underline v)$ is finite \'etale, the proof proceeds exactly as in Proposition~\ref{stalkJs}, taking the fibers of the above diagrams at the point corresponding to $x$. 
(Note that the stalks of the structural sheaf of a rigid variety are Henselian \cite[Lemma 2.1.1]{deJongderPut}.)

To conclude, if the Fourier--Jacobi expansion at $x$ of a form vanishes, it must vanish in an open neighbourhood of  $x$, and hence on the whole connected component containing $x$.
\end{proof}
\subsection{\texorpdfstring{$p$-adic Siegel morphism}{p-adic Siegel morphism}} \label{subsec: p-adic Siegel morphism}
We now study the Siegel morphism for families of $p$-adic modular forms.

Let $\Spm(A) \subset \mc W_{a}(\underline w)^\circ$ or $\Spm(A) \subset \mc W_{b,a}(\underline w)^\circ$ be a fixed affinoid admissible open of the weight space. With a little abuse of notation we continue to write $\eta$ for the morphism
\[
\eta \colon \mathfrak S^{\tor, \rig}_G(\mc H p^n)(\underline v) \times \Spm(A) \to \mathfrak S^{\ast, \rig}_G(\mc H)(\underline v) \times \Spm(A).
\]
Let $q=0,\ldots, a$ be an integer. We have the closed immersion
\[
\bigcup_{s=q+1}^a \bigcup_{[\gamma] \in C_s} \mathfrak S_{G_s}^{\rig}(\mc H_{[\gamma]})(\underline v) \times \Spm(A) \hookrightarrow \mathfrak S_{G}^{\ast,\rig}(\mc H)(\underline v) \times \Spm(A)
\]
given by the stratification of the minimal compactification. We will write $\mc I^q$ for the corresponding sheaf of ideals and we define $\mc J^q \colonequals \eta^\ast \mc I^q$. By definition $\mc J^q$ is the ideal defined by the coarse stratification of the boundary of the toroidal stratification. We now describe the stalks of the sheaf $\eta_\ast ( \underline \omega^{\dagger\chi}_{\underline v,\underline w} \otimes \mc J^q)$, as in Proposition~\ref{prop: stalk min circ}, using the same notation.
\begin{prop}\label{prop: stalk min circ with cusp}
Let $\mc U$ be an open affinoid in $\mc W_{b,a}(\underline w)^\circ$ or $\mc W_a(\underline w)^\circ$, with universal character $\chi_{\mc U}^{\un}$. Let $[\gamma] \in C_s(\mc H)$ be a cusp label and let $x$ be a geometric adic point of $\mathfrak S_{G_s}^{\mr{ad}}(\mc H_{[\gamma]})(\underline v)$. The completion of the stalk of $\eta_\ast (\underline \omega^{\dagger\chi_{\mc U}^{\un}}_{\underline v,\underline w} \otimes \mc J^q)$ is canonically isomorphic to 
\begin{align*}
\prod_{[\gamma']}\left(\prod_{[h] \in H^+_{[\gamma']}/\G_{[\gamma']}, \mr{rk}(h)\geq s- q} {\Homol^0(\hat{\mathfrak Z}^{\mr{ad}}_{[\gamma'],x},\mc L(h)\otimes\underline \omega^{\dagger\chi_{\mc U}^{\un}}_{\underline v,\underline w})}^{\G_{[\gamma']}(h)}\right).
\end{align*}
An analogous result holds if $\mc U$ is in $\mc W_{b,a}(\underline w)$ or $\mc W_a(\underline w)$.
\end{prop}
Let $\iota_{[\gamma]} \colon \mathfrak S_{G_s}^{\ast, \rig}(\mc H_{[\gamma]})(\underline v) \times \Spm(A) \to \mathfrak S_{G}^{\ast,\rig}(\mc H)(\underline v) \times \Spm(A)$ be the natural morphism. In the following lemma $\mc I^{q-1}_{[\gamma]}$ has the same definition as $\mc I^{q-1}$, but for the Shimura variety $\mathfrak S_{G_s}^{\ast, \rig}(\mc H_{[\gamma]})(\underline v)$.
\begin{lemma} \label{lemma: ex sec}
There is an exact sequence of sheaves on $\mathfrak S^{\ast, \rig}_G(\mc H)(\underline v) \times \Spm(A)$
\begin{equation} \label{eq: ex sec first}
0 \to \eta_\ast(\underline \omega_{\underline v,\underline w}^{\dagger\chi_{\mc U}^{\un}} \otimes \mc J^0) \to \eta_\ast( \underline \omega_{\underline v,\underline w}^{\dagger\chi_{\mc U}^{\un}} \otimes \mc J^q) \to \bigoplus_{[\gamma] \in C_1} \eta_\ast(\underline \omega_{\underline v,\underline w}^{\dagger\chi_{\mc U}^{\un}} \otimes  \eta^\ast \iota_{[\gamma],\ast} \mc I^{q-1}_{[\gamma]})
\end{equation}
\end{lemma}
\begin{proof}
Using exact sequence
\[
0 \to \mc I^0 \to \mc I^q \to \bigoplus_{[\gamma] \in C_1} \iota_{[\gamma],\ast} \mc I^{q-1}_{[\gamma]} \to 0
\]
we obtain the exact sequence, on $\mathfrak S^{\tor, \rig}_G(\mc H)(\underline v) \times \Spm(A)$,
\[
0 \to \mc J^0 \to \mc J^q \to \bigoplus_{[\gamma] \in C_1} \eta^\ast \iota_{[\gamma],\ast} \mc I^{q-1}_{[\gamma]} \to 0
\]
We can now tensor with $\underline \omega_{\underline v,\underline w}^{\dagger\chi_{\mc U}^{\un}}$ and push-forward to $\mathfrak S^{\ast, \rig}_G(\mc H)(\underline v) \times \Spm(A)$ via $\eta$. The only non trivial thing to check is exactness on the left, that can be checked on completed stalks using Proposition~\ref{prop: stalk min circ with cusp}.
\end{proof}
\begin{defi}
Let $\mc U = \Spm(A)$ and $q$ be as above. Let $\chi_{\mc U}^{\un}$ be the universal character associated to $\mc U$. We define
\[
\M^q_{\chi_{\mc U}^{\un}}(\mc Hp^n, K) \colonequals \Homol^0(\mathfrak S^{\ast, \rig}_G(\mc H)(\underline v) \times \Spm(A), \eta_\ast(\underline \omega_{\underline v,\underline w}^{\dagger\chi_{\mc U}^{\un}}\otimes \mc J^q)).
\]
Given a modular form $f$ in $\M^a_{\chi_{\mc U}^{\un}}(\mc Hp^n, K)$ (that is just a global section of $\underline \omega_{\underline v,\underline w}^{\dagger\chi_{\mc U}^{\un}}$) with $f \neq 0$, we define its corank $\cork(f)$ as the smallest integer $q$ such that $f \in \M^q_{\chi_{\mc U}^{\un}}(\mc Hp^n, K)$.
\end{defi}
Given a positive integer $s=1,\ldots,a$ and a weight $\chi$ (not necessarily integral) we define the notion of $\chi$ being of corank $s$ in the obvious way, exactly as in the case of integral weights. This gives us the closed subspaces $\mc W_{a}(\underline w)^{\circ,s} \subseteq \mc W_{a}(\underline w)^\circ$ and $\mc W_{a,b}(\underline w)^{\circ,s} \subseteq \mc W_{b,a}(\underline w)^\circ$ given by weights of corank at least $s$. It is convenient to set $\mc W_{a}(\underline w)^{\circ,0} = \mc W_{a}(\underline w)^{\circ}$ and $\mc W_{b,a}(\underline w)^{\circ,0} = \mc W_{b,a}(\underline w)^{\circ}$. If $s > 0$, we have that $\mc W_{a}(\underline w)^{\circ,s}$ and $\mc W_{b,a}(\underline w)^{\circ,s}$ are weight spaces in $d(a-s+1)$ variables (recall that $d$ is the degree of the relevant totally real field). The map $k \mapsto k'$ defined at the beginning of Subsection~\ref{subsec: sieg morph} extends to a morphism $\cdot ' \colon \mc W_{a}(\underline w)^{\circ,s} \to \mc W_{a-1}(\underline w)^{\circ,s-1}$ or $\cdot ' \colon \mc W_{b,a}(\underline w)^{\circ,s} \to \mc W_{b-1,a-1}(\underline w)^{\circ,s-1}$. We will use an analogous notation for the weight spaces $\mc W_{a}(\underline w)^{s} \subseteq \mc W_{a}(\underline w)$ and $\mc W_{b,a}(\underline w)^{\circ,s} \subseteq \mc W_{b,a}(\underline w)$ of level $\mc H$.

Let $\mc U = \Spm(A)$ be an affinoid admissible open of $\mc W_{a}(\underline w)^{\circ,s}$ or $\mc W_{a,b}(\underline w)^{\circ,s}$, with corresponding universal character $\chi_{\mc U}^{\un}$. Over $\mathfrak S_G^{\tor,\rig}(\mc Hp^n)(\underline v) \times \Spm(A)$ we can define the sheaf $\underline \omega_{\underline v,\underline w}^{\dagger \chi_{\mc U}^{\un}}$. All the results proved for the modular sheaves stay true for $\underline \omega_{\underline v,\underline w}^{\dagger \chi_{\mc U}^{\un}}$. In particular we have the formal model $\underline {\mathfrak w}_{\underline v,\underline w}^{\dagger \chi_{\mathfrak U}^{\un}}$, where $\mathfrak U$ is a formal model of $\mc U$.

Let $s=1,\ldots,a$ be a positive integer and let $[\gamma] \in C_1$ be a cusp label. Let $\mc U = \Spm(A)$ be an affinoid admissible open of $\mc W_{a}(\underline w)^{\circ,s}$ or $\mc W_{b,a}(\underline w)^{\circ,s}$, with corresponding universal character $\chi_{\mc U}^{\un}$. We write $\mc V=\Spm(B)$, with universal character $\chi_{\mc V}^{\un}$, for the affinoid admissible open of $\mc W_{a-1}(\underline w)^{\circ,s-1}$ or $\mc W_{b-1,a-1}(\underline w)^{\circ,s-1}$ given by the image of $\mc U$ via $\chi \mapsto \chi'$. We write $\iota_{[\gamma]} \colon \mathfrak S_{G_1}^{\ast,\rig}(\mc H_{[\gamma]})(\underline v) \times \Spm(A) \to \mathfrak S_{G}^{\ast,\rig}(\mc H)(\underline v) \times \Spm(A)$ for the natural morphism. We let $\eta_A$ and $\eta_B$ be the morphisms
\begin{gather*}
\eta_A \colon \mathfrak S_{G}^{\tor,\rig}(\mc Hp^n)(\underline v) \times \Spm(A)\to \mathfrak S_{G}^{\ast,\rig}(\mc H)(\underline v) \times \Spm(A), \\
\eta_B \colon \mathfrak S_{G_1}^{\tor,\rig}(\mc H_{[\gamma]}p^n)(\underline v) \times \Spm(B) \to \mathfrak S_{G_1}^{\ast,\rig}(\mc H_{[\gamma]})(\underline v) \times \Spm(B)
\end{gather*}
\subsubsection*{Some representation theory} Before proving the key proposition on the overconvergent Siegel morphism, we need to generalise certain results of representation theory of $\GL$ to analytic representations of the Iwahori subgroup. First of all some notation.

Let $0\leq r \leq a$ and define $\Sop^{\mc O,r,\op}$ to be, in the unitary case, the set of couples of matrices in $\N^{\mc O,\op}$ such that the non-zero elements of the first component are in the lower left $(b-r) \times r$-block and the non-zero elements of the second component are in upper right $(a-r) \times r$-block. The definition of $\Sop^{\mc O, r, \op}$ in the symplectic case is similar considering the lower $r \times a-r$-block. Recall the various notation, as for example $\T^{\mc O,b,a}$ and $\T^{\mc O,r}$, introduced before Proposition~\ref{prop: omega circ repr gen}.
\begin{lemma} \label{lemma: invariants}
Let $\mc U = \Spm(A)$. In the unitary case, write $\chi_{\mc U}^{\un}$ as $\chi_1 \chi_2 \chi_3$ according to the decomposition $\T^{\mc O,b,a}=\T^{\mc O,r}\times \T^{\mc O,b-r,a-r}\times \T^{\mc O,r}$, such that each $\chi_i$ is $\underline w_i$-analytic in the obvious sense. We have then
\[
{(L_{\chi_{\mc U}^{\un}}^{\underline w-\an, \circ})}^{\Sop^{\mc O, r,\op}} = L_{\chi_{1}}^{\underline w_1-\an, \circ}\hat{\otimes}_{A}L_{\chi_{2}}^{\underline w_2-\an, \circ}\hat{\otimes}_{A}L_{\chi_{3}}^{\underline w_3-\an, \circ} 
\]
where the right hand side is a representation of $\I^{\mc O,r} \times \I^{\mc O,b-r,a-r} \times \I^{\mc O,r,\op} $. In the symplectic case, write $\chi_{\mc U}^{\un}$ as $\chi_1 \chi_2$ according to the decomposition $\T^{\mc O,a}=\T^{\mc O,r}\times \T^{\mc O,a-r}$, we have then
\[
{(L_{\chi_{\mc U}^{\un}}^{\underline w-\an, \circ})}^{\Sop^{\mc O, r,\op}} = L_{\chi_{1}}^{\underline w_1-\an, \circ}\hat{\otimes}_{A}L_{\chi_{2}}^{\underline w_2-\an, \circ}
\]
as a representation of $\I^{\mc O,r} \times \I^{\mc O,a-r} $.
\end{lemma}
\begin{proof}
Consider a function $f \in \mc F^{\underline w-\an}(\N^{\mc O, \op},A)$. Begin invariant by $\Sop^{\mc O, r,\op}$ means that $n \star f = f$. In particular $n \star f (1) = f(1n)=f(1)$. Hence the map 
\begin{gather*}
{(L_{\chi_{\mc U}^{\un}}^{\underline w-\an, \circ})}^{\Sop^{\mc O, r,\op}} \rightarrow \\
\mc F^{\underline w_1-\an}(\N^{\mc O,r, \op},A) \times \mc F^{\underline w_2-\an}(\N^{\mc O, b-r,a-r,\op},A) \times \mc F^{\underline w_3-\an}(\N^{\mc O,r},A)
\end{gather*}
is injective and hence surjective for dimension reasons and it respects the action of $\I^{\mc O,r} \times \I^{\mc O,b-r,a-r} \times \I^{\mc O,r,\op} $. The proof in the symplectic case is similar.
\end{proof}

We are now ready to define and study the Siegel morphism.
\begin{prop} \label{prop: identification sheaves}
Let $q=1,\ldots,a$ be an integer with $q \leq s$. If $q \neq 1$ we have a natural isomorphism of sheaves on $\mathfrak S_{G_1}^{\ast,\rig}(\mc H)(\underline v) \times \Spm(B)$
\[
\iota_{[\gamma]}^\ast \left ( \eta_{A,\ast} ( \underline \omega_{\underline v,\underline w}^{\dagger \chi_{\mc U}^{\un}} \otimes \mc J^q ) \right) \cong \eta_{B,\ast} (\underline \omega_{\underline v,\underline w}^{\dagger \chi_{\mc V}^{\un}} \otimes \mc J^{q-1}_{[\gamma]}).
\]
Moreover, we have a natural isomorphism of sheaves on $\mathfrak S_{G_1}^{\ast,\rig}(\mc H)(\underline v) \times \Spm(B)$
\[
\iota_{[\gamma]}^\ast \left ( \eta_{A,\ast}( \underline \omega_{\underline v,\underline w}^{\dagger \chi_{\mc U}^{\un}} \otimes \mc J^1 ) \right) \cong \eta_{B,\ast} ( \underline \omega_{\underline v,\underline w}^{\dagger \chi_{\mc V}^{\un}} \otimes \mc J^{0}_{[\gamma]} ) \otimes \mc O_{\mathbb B(0,1)_B},
\]
where $\mathbb B(0,1)_B$ is the closed unit ball of radius $1$ over $B$.
\end{prop}
\begin{proof}
First of all we compare the completion of the stalks of the two sheaves at any geometric adic point of $\mathfrak S_{G_s}^{\ad}(\mc Hp^n_{[\gamma]})(\underline v)$. Let $x$ be such a point. As it will not affect the idea of the proof, we shall ignore the morphism $\xi$ and we suppose that our sheaves are already defined on the toroidal compactification of the variety without level at $p$. Recall the notation from Proposition~\ref{prop: SiegelSheaves}: we identify the set of cusp labels $[\gamma_1]$ for $\mathfrak S_{G}^{\ast,\rig}(\mc H)(\underline v)$ which belong to the closure of $\mathfrak S_{G_1}^{\rig}(\mc H_{[\gamma]})(\underline v)$ with the set of cusp label $[\gamma_2]$ for $\mathfrak S_{G_1}^{\ast,\rig}(\mc H_{[\gamma]})(\underline v)$. For each pair of corresponding cusp labels $[\gamma_1]$ and $[\gamma_2]$, using the description of the Fourier--Jacobi expansion given in Proposition~\ref{prop: stalk min circ with cusp} and that the ideal of the boundary is generated by the $q^h$ for $h$ with maximal rank, we obtain:
\begin{align*}
\widehat{\iota_{[\gamma]}^\ast \eta_{A,\ast} \left (  \underline \omega_{\underline v,\underline w}^{\dagger \chi_{\mc U}^{\un}} \otimes \mc J^q \right)}_x = &\prod_{[h] \in H^+_{[\gamma_1]}/\G_{[\gamma_1]}, s>\mr{rk}(h)\geq s- q} {\Homol^0(\hat{\mathfrak Z}^{\mr{ad}}_{[\gamma_1],x},\mc L(h)\otimes\underline \omega^{\dagger\chi_{\mc U}^{\un}}_{\underline v,\underline w})}^{\G_{[\gamma_1]}(h)},\\
\eta_{B,\ast} \widehat{( \underline \omega_{\underline v,\underline w}^{\dagger \chi_{\mc V}^{\un}} \otimes \mc J^{q-1}_{[\gamma]})}_{x}  = & \prod_{[h] \in H^+_{[\gamma_2]}/\G_{[\gamma_2]}, \mr{rk}(h)\geq s-q} {\Homol^0(\hat{\mathfrak Z}^{\mr{ad}}_{[\gamma_2],x},\mc L(h)\otimes\underline \omega^{\dagger\chi_{\mc U}^{\un}}_{\underline v,\underline w})}^{\G_{[\gamma_2]}(h)}.
\end{align*}
We recall from Proposition~\ref{prop: stalk min circ}:
\[
{\Homol^0(\hat{\mathfrak Z}^{\mr{ad}}_{[\gamma_1],x},\mc L(h)\otimes\underline \omega^{\dagger\chi_{\mc U}^{\un}}_{\underline v,\underline w})} \cong \Homol^0(\hat{\mathfrak Z}^{\mr{ad}}_{[\gamma_1],x},\mc L(h))\hat{\otimes}L^{\underline w-\an, \circ}_{-w_0 \chi_{\mc U}^{\un}}.
\]
Remark that for each cusp label $[\gamma_1] \in C_r(\mc H)$ and $h  \in H^+_{[\gamma_1]}$ (notation as in the proof of  {\it loc. cit.}) the group $\G_{[\gamma_1]}(h)$ (even at $p^n$-full level) is dense in $\Sop^{\mc O,s,\op}$, which allows one to use Lemma~\ref{lemma: invariants}. Hence as in Proposition~\ref{prop: SiegelSheaves} we have 
\begin{align*}
{\Homol^0(\hat{\mathfrak Z}^{\mr{ad}}_{[\gamma_1],x},\mc L(h)\otimes\underline \omega^{\dagger\chi_{\mc U}^{\un}}_{\underline v,\underline w})}^{\G_{[\gamma_1]}(h)} = & {\Homol^0(\hat{\mathfrak Z}^{\mr{ad}}_{h,x},\tilde{\mc L}(h)\otimes\underline \omega^{\dagger\chi_{\mc U}^{\un}}_{\underline v,\underline w})}^{\G_{[\gamma_1]}(h)} \\
= & {\Homol^0(\hat{\mathfrak Z}^{\mr{ad}}_{h,x},\tilde{\mc L}(h)\otimes\underline \omega^{\dagger\chi_{\mc V}^{\un}}_{\underline v,\underline w})}^{\G_{[\gamma_2]}(h)} \\
= & {\Homol^0(\hat{\mathfrak Z}^{\mr{ad}}_{[\gamma_2],x},\mc L(h)\otimes\underline \omega^{\dagger\chi_{\mc V}^{\un}}_{\underline v,\underline w})}^{\G_{[\gamma_2]}(h)}.
\end{align*}
Hence the completions of the two stalks are the same. 

We now explain why the stalks are the same, without taking the completion. First of all note that if $x$ lies in the open part $\mathfrak S_{G_1}^{\ad}(\mc H_{[\gamma]})(\underline v)$ of our cusp then the above products reduce to the single term corresponding to $h = 0$. 

Given an element of the stalk of ${\eta_{B,\ast} ( \underline \omega_{\underline v,\underline w}^{\dagger \chi_{\mc V}^{\un}} \otimes \mc J^{q-1}_{[\gamma]}})_{x}$ it is just a section of 
\[
{\Homol^0({\mathfrak Z}^{\mr{ad}}_{[\gamma_2]} \times {\mathfrak{U}},\underline \omega^{\dagger\chi_{\mc V}^{\un}}_{\underline v,\underline w})}^{\G_{[\gamma_2]}(h)}, 
\]
for $\mathfrak{U}$ a small enough open neighbourhood of $x$. Using the above chain of isomorphisms, one can define an element of ${\Homol^0({\mathfrak Z}^{\mr{ad}}_{[\gamma_1]}\times \mathfrak{U},\underline \omega^{\dagger\chi_{\mc U}^{\un}}_{\underline v,\underline w})}^{\G_{[\gamma_1]}(h)}$ which is then an element in the stalk $\iota_{[\gamma]}^\ast \eta_{A,\ast} ( \underline \omega_{\underline v,\underline w}^{\dagger \chi_{\mc U}^{\un}} \otimes \mc J^q)_x $. Note also that they define overconvergent forms on the open Shimura variety of smaller genus by the same reason as in Remark \ref{rem:FaltingSiegel}.

For general $x$, any $f$ in $\eta_{B,\ast} ( \underline \omega_{\underline v,\underline w}^{\dagger \chi_{\mc V}^{\un}} \otimes \mc J^{q-1}_{[\gamma]})_{x}$ is defined over suitable open $\mathfrak{U}$. Let $y$ be a point of $\mathfrak{U}$ lying in the open part of the boundary. As in the previous step, we can define 
\[
f_{\mathfrak{U}} \in  {\Homol^0({\mathfrak Z}^{\mr{ad}}_{[\gamma_2]} \times {\mathfrak{U}},\underline \omega^{\dagger\chi_{\mc V}^{\un}}_{\underline v,\underline w})}^{\G_{[\gamma_2]}(h)},
\]
which corresponds to
\[
\tilde{f}_{\mathfrak{U}} \in {\Homol^0({\mathfrak Z}^{\mr{ad}}_{[\gamma_1]}\times \mathfrak{U},\underline \omega^{\dagger\chi_{\mc U}^{\un}}_{\underline v,\underline w})}^{\G_{[\gamma_1]}(h)},
\]
which is a overconvergent modular forms on the open variety. We have to show that $\tilde{f}_{\mathfrak{U}}$ extend to the whole $\mathfrak{U}$, so that it will be an element of $\iota_{[\gamma]}^\ast \eta_{A,\ast} ( \underline \omega_{\underline v,\underline w}^{\dagger \chi_{\mc U}^{\un}} \otimes \mc J^q)_x $. 
But, using the isomorphism on the completions, we know that Fourier--Jacobi expansion $\tilde{f}$ is the one corresponding to $f$. But $f$ has no negative definite terms, hence so does $\tilde{f}$, which extends.
This completes the proof.
\end{proof}
\begin{rmk} \label{rmk: reason prod tens}
The reason for the appearance of the tensor product in the case $q=1$ is intuitively the following. Let us suppose for simplicity the we are in the symplectic case. A $p$-adic weight $\chi$ is of corank $1$ is and only if its last component is parallel and the morphism $\chi \mapsto \chi'$ forgets this last component. In particular we can recover from $\chi'$ all the components of $\chi$ but the last one, that is parallel: this gives the extra variable. If $q>1$ then we can recover $\chi$ from $\chi'$.
\end{rmk}
Let $q > 1$. For every $[\gamma] \in C_1$, we have isomorphisms
\begin{gather*}
\eta_{A,\ast}(\underline \omega_{\underline v,\underline w}^{\dagger\chi_{\mc U}^{\un}} \otimes \eta^\ast \iota_{[\gamma],\ast} \mc I^{q-1}_{[\gamma]}) \cong \iota_{[\gamma],\ast} \iota_{[\gamma]}^\ast (\eta_{A,\ast}(\underline \omega_{\underline v,\underline w}^{\dagger\chi_{\mc U}^{\un}} \otimes  \eta^\ast \iota_{[\gamma],\ast} \mc I^{q-1}_{[\gamma]})) \cong \\
\iota_{[\gamma],\ast} \iota_{[\gamma]}^\ast (\eta_{A,\ast} (\underline \omega_{\underline v,\underline w}^{\dagger \chi_{\mc U}^{\un}} \otimes \mc J^q)).
\end{gather*}
Using exact sequence \eqref{eq: ex sec first} and Proposition~\ref{prop: identification sheaves} we get an exact sequence
\begin{equation} \label{eq: ex sec sec}
0 \to \eta_{A,\ast}(\underline \omega_{\underline v,\underline w}^{\dagger\chi_{\mc U}^{\un}} \otimes \mc J^0) \to \eta_{A,_\ast}( \underline \omega_{\underline v,\underline w}^{\dagger\chi_{\mc U}^{\un}} \otimes \mc J^q) \to \bigoplus_{[\gamma] \in C_1} \iota_{[\gamma],\ast} \eta_{B,\ast} (\underline \omega_{\underline v,\underline w}^{\dagger \chi_{\mc V}^{\un}} \otimes \mc J^{q-1}_{[\gamma]})
\end{equation}
and we have an analogous result if $q=1$.

Taking global sections we get morphisms
\begin{equation} \label{eq: sieg morph q}
\M^q_{\chi_{\mc U}^{\un}}(\mc Hp^n,K) \to \bigoplus_{[\gamma] \in C_1} \M^{q-1}_{\chi_{\mc V}^{\un}}(\mc H_{[\gamma]}p^n,K)
\end{equation}
and
\begin{equation} \label{eq: sieg morph 1}
\M^1_{\chi_{\mc U}^{\un}}(\mc Hp^n,K) \to \bigoplus_{[\gamma] \in C_1} \M^{0}_{\chi_{\mc V}^{\un}}(\mc H_{[\gamma]}p^n,K) \otimes_B B\langle x \rangle.
\end{equation}
These are the so called $p$-adic Siegel morphisms. We are going to show that they are surjective.
\begin{prop} \label{prop: ex seq sieg}
Let $q$, $\Spm(A)$, and $\Spm(B)$, etc be as above. We have that exact sequence \eqref{eq: ex sec sec} is exact on the right, giving, if $q > 1$, the exact sequence of sheaves
\[
0 \to \eta_{A,\ast} (\underline \omega_{\underline v,\underline w}^{\dagger \chi_{\mc U}^{\un}} \otimes \mc J^0) \to \eta_{A,\ast}( \underline \omega_{\underline v,\underline w}^{\dagger \chi_{\mc U}^{\un}} \otimes \mc J^q) \to \bigoplus_{[\gamma] \in C_1} \iota_{[\gamma],\ast} \left( \eta_{B,\ast} (\underline \omega_{\underline v,\underline w}^{\dagger \chi_{\mc V}^{\un}} \otimes \mc J^{q-1}_{[\gamma]}) \right) \to 0
\]
If $q=1$ we have an analogous result taking into account the extra factor $\mc O_{\mathbb B(0,1)_B}$ in the last morphism.
\end{prop}
\begin{proof}
Taking into account Proposition~\ref{prop: stalk min circ with cusp}, this is again a computation using Fourier--Jacobi expansion as in the proof of Proposition~\ref{prop: identification sheaves}.
\end{proof}
Let $\mathfrak X$ be a flat, integral, normal, quasi-projective formal scheme over $\Spf(\mc O_K)$ that is topologically of finite type and such that its rigid fiber is an affinoid. In the appendix of \cite{AIP} the notion of a formal flat Banach sheaf over $\mathfrak{X}$ is defined. Let $\mathfrak F$ be such a sheaf. By \cite[Section~A.2.2]{AIP} we can attach to $\mathfrak F$ a Banach sheaf $\mc F= \mathfrak F^{\rig}$ on $\mathfrak X^{\rig}$. We now prove a general results about the vanishing of the cohomology of a small, formal, and flat Banach sheaf (see \cite[Definition~A.1.2.1]{AIP} for the definition of a small formal Banach sheaf).
\begin{prop} \label{prop: coho banach 0}
Let $\mathfrak X$ and $\mathfrak F$ be as above and suppose moreover that $\mathfrak F$ is small. Let $\mc X \colonequals \mathfrak X^{\rig}$ and $\mc F \colonequals \mathfrak F^{\rig}$. Then
\[
\Homol^1(\mc X, \mc F) = 0.
\]
\end{prop}
\begin{proof}
Recall that the natural morphism from sheaf cohomology to Čech cohomology is always (on any site) an isomorphism in degree $1$, see for example \cite[Corollary~3.4.7]{tamme}. In particular it is enough to prove that
\[
\check{\Homol}^1(\mc X, \mc F) = 0.
\]
The open covers of $\mc X$ made by affinoids are a base of the topology, in particular we can prove that
\[
\check{\Homol}^1(\set{\mc U_i}_{i \in I}, \mc F) = 0,
\]
where $\set{\mc U_i}_{i \in I}$ is an open cover of $\mc X$ and each $\mc U_i$ is an affinoid. By quasi-compactness of $\mc X$ we can assume $I$ to be finite. Taking a finite number of admissible blow-ups of $\mathfrak X$ we get a formal schemes $\tilde{\mathfrak X}$ such that the cover $\set{\mc U_i}_{i \in I}$ comes from a cover $\set{\mathfrak U_i}_{i \in I}$ of $\tilde{\mathfrak X}$ made by affine formal schemes. Note that $\tilde{\mathfrak X}$ satisfies the same properties of $\mathfrak X$ and $\tilde{\mathfrak X}^{\rig}=\mc X$. Since $\mathfrak F$ is a flat formal Banach sheaf, we have, by \cite[Lemma~A.2.2.2]{AIP},
\[
\mc F(\mc U_i) = \mathfrak F(\mathfrak U_i)[1/p]
\]
for all $i \in I$. We conclude since, by \cite[Theorem~A.1.2.2]{AIP} the Čech complex associated to $\set{\mathfrak U_i}_{i \in I}$ is exact.
\end{proof}
The morphism $\eta$ extends to a morphism between the formal models, denoted with the same symbol
\[
\eta \colon \mathfrak S_G^{\tor}(\mc Hp^n)(\underline v) \to \mathfrak S_G^{\ast}(\mc H)(\underline v).
\]
We write $\mathfrak I^0$ for the ideal corresponding to the boundary of $\mathfrak S_G^{\ast}(\mc H)(\underline v) \times \mathfrak U$, where $\mathfrak U$ is a formal model of $\mc U$.
\begin{coro} \label{coro: sieg morph surj}
The $p$-adic Siegel morphisms \eqref{eq: sieg morph 1} and \eqref{eq: sieg morph q} are surjective.
\end{coro}
\begin{proof}
Using Proposition~\ref{prop: ex seq sieg}, it is enough to prove that
\[
\Homol^1(\mathfrak S_G^{\ast,\rig}(\mc H)(\underline v),\eta_{A,\ast} \underline \omega_{\underline v,\underline w}^{\dagger \chi_{\mc U}^{\un}} \otimes \mc I^0) = 0.
\]
We have that $\mathfrak S_G^{\ast}(\mc H)(\underline v)$ is a flat, integral, normal, quasi-projective formal scheme over $\Spf(\mc O_K)$ that is topologically of finite type and such that its rigid fiber is an affinoid. Moreover $\eta_\ast \underline {\mathfrak w}_{\underline v,\underline w}^{\dagger \chi_{\mathfrak U}^{\un}} \otimes \mathfrak I^0$ is a flat formal Banach sheaf over $\mathfrak S_G^{\ast}(\mc H)(\underline v)$ that is small by \cite[Section~8 and Appendix]{AIP}. The corollary follows by Proposition~\ref{prop: coho banach 0}.
\end{proof}
Let $q,s = 1, \ldots, a$ be integers with $q \leq s$. Let $\Spm(A)$ be an affinoid admissible open of $\mc W_a(\underline w)^{\circ, s}$ or $\mc W_{b,a}(\underline w)^{\circ, s}$ with universal character $\chi_{\mc U}^{\un}$. Let $\Spm(B)=\mc V$ be the image of $\mc U$ under $\cdot'$ and let $\chi_{\mc V}^{\un}$ be the universal character of $\mc V$. We have proved that there is an exact sequence, if $q>1$,
\[
0 \to \M^0_{\chi_{\mc U}^{\un}}(\mc Hp^n,K) \to \M^q_{\chi_{\mc U}^{\un}}(\mc Hp^n,K) \to \bigoplus_{[\gamma] \in C_1} \M^{q-1}_{\chi_{\mc V}^{\un}}(\mc H_{[\gamma]}p^n,K) \to 0
\]
If $q=1$ we have the exact sequence
\[
0 \to \M^0_{\chi_{\mc U}^{\un}}(\mc Hp^n,K) \to \M^1_{\chi_{\mc U}^{\un}}(\mc Hp^n,K) \to \bigoplus_{[\gamma] \in C_1} \M^0_{\chi_{\mc V}^{\un}}(\mc H_{[\gamma]}p^n,K) \otimes_B B \langle x \rangle \to 0
\]
To obtain a similar result for level $\mc H$ (and the weight spaces $\mc W_a(\underline w)$ or $\mc W_{b,a}(\underline w)$) it is enough to take invariant for the action of the group $\B^{\mc O}(\Z/p^n\Z)$ (this groups takes into account the difference between level $\mc Hp^n$ and $\mc H$). Indeed, since we are in characteristic $0$, the group $\B^{\mc O}(\Z/p^n\Z)$ has no cohomology and we obtain the following theorem.
\begin{teo} \label{teo: fund ex sec}
Let $q,s = 1, \ldots, a$ be integers with $q \leq s$. Let $\Spm(A)$ be an affinoid admissible open of $\mc W_a(\underline w)^{s}$ or $\mc W_{b,a}(\underline w)^{s}$ with universal character $\chi_{\mc U}^{\un}$. Let $\Spm(B)=\mc V$ be the image of $\mc U$ under $\cdot'$ and let $\chi_{\mc V}^{\un}$ be the universal character of $\mc V$. If $q>1$ there is an exact sequence
\begin{equation} \label{eq: fund ex sec}
0 \to \M^0_{\chi_{\mc U}^{\un}}(\mc H,K) \to \M^q_{\chi_{\mc U}^{\un}}(\mc H,K) \to \bigoplus_{[\gamma] \in C_1} \M^{q-1}_{\chi_{\mc V}^{\un}}(\mc H_{[\gamma]},K) \to 0.
\end{equation}
If $q=1$ we have the exact sequence
\begin{equation} \label{eq: fund ex sec 1}
0 \to \M^0_{\chi_{\mc U}^{\un}}(\mc H,K) \to \M^1_{\chi_{\mc U}^{\un}}(\mc H,K) \to \bigoplus_{[\gamma] \in C_1} \M^0_{\chi_{\mc V}^{\un}}(\mc H_{[\gamma]},K) \otimes_B B \langle x \rangle \to 0.
\end{equation}
\end{teo}

\section{Eigenvarieties for non-cuspidal systems of automorphic forms} \label{sec: eigen}
In this section $\mc H$ is a compact open subgroup of $G(\m A_{F_0,f})$, not necessarily neat and $p$ satisfies the hypotheses of Section \ref{sec: p-adic Section}.
\subsection{The eigenvariety machinery} \label{subsec: eigen machinery}
We briefly recall Buzzard's eigenvarieties machinery  \cite{buzz_eigen}.
Let $A$ be a Noetherian Banach algebra over $\Q_p$ and $M$ a Banach module over $A$. In particular $M$ is equipped with a norm $\vert \phantom{a} \vert$. We say that another norm $\vert \phantom{a} \vert'$ is equivalent to $\vert \phantom{a} \vert$ if they induce the same topology.
\begin{defin}
We say that $M$ is $\mr{Pr}$ if there exist a Banach module $N$ and an equivalent norm on $M \oplus N$ such that \begin{align*}
M \oplus N \cong \widehat{\bigoplus_I} A 
\end{align*}
as Banach modules.
\end{defin}
The term $\mr{Pr}$ could remind the reader of projective module. But this is misleading: even if a $\mr{Pr}$ module satisfies the universal property of lifting surjective morphism it is not necessarily projective. This is because in the category of Banach $A$-modules epimorphisms are the morphisms with dense image.

\begin{defin} Let $U$ be a continuous $A$-linear operator on $M$, we say that is completely continuous (or compact) if $U$ can be written as a limit (for the operator norm) of continuous operators of finite rank.
\end{defin}

 We are now ready to recall the eigenvariety machinery. We are given as input:
\begin{itemize}
\item a reduced, equidimensional affinoid $\mathrm{Spm}(A)$,
\item a $\mr{Pr}$ module $M$ over $A$,
\item a commutative endomorphism algebra $\mathbf{T}$ of $M$ over $A$,
\item a compact operator $U$ of $\mathbf{T}$.
\end{itemize}
With these objects we can define a formal series $P(T) \colonequals \det(1 -TU \vert M) \in A\left\{\left\{ T \right\}\right\}$. 
\begin{defi} 
The spectral variety $\mathcal{Z}$ associated to $(A,M,U)$ is the closed subspace of $\mathrm{Spm}(A) \times \mathbb{A}^1$ defined by $P(T)=0$.
\end{defi}
A point $(x,\lambda)$ belongs to $\mathcal{Z}$ if and only if there exists $m \in M\otimes_A \overline{\kappa(x)}$ such that $U(m)= \lambda^{-1}m$.

To the above data Buzzard associates a rigid analytic space $\mr{pr}:\mc E \rightarrow \mathcal{Z}$ with an admissible cover $\mc E_{\alpha}$ which satisfies the following properties:
\begin{itemize}
\item $\mc E$ is equidimensional of dimension $\mr{dim}(A)$.

\item Let $\kappa$ be the structural morphism $\mathcal{E} \rightarrow \mathrm{Spm}(A)$. The space $\mathcal{E}$ parametrizes system of eigenvalues appearing in $M$; indeed, each point in $\kappa^{-1}(x)$ corresponds a system of eigenvalues for $\mathbb{T}$ inside $M\otimes_A \overline{\kappa(x)}$ which is of finite slope for $U$.

\item The map $\mr{pr}$ is finite. The map $\kappa$ is locally finite.

\item The module $M$ defines a coherent sheaf $\mathcal{M}$. The fiber $\mathcal{M}_{(x,\lambda)}$ is the generalised eigenspace for $\lambda^{-1}$ inside $M\otimes_A \overline{\kappa(x)}$. 

\item The image of  $\mc E_{\alpha} $ in $\mr{Spm}(A)$ is affinoid. 

\item Over $\kappa(\mc {E}_{\alpha})$ we have a factorisation $P(T)=P_1(T)P_2(T)$ with $P_1(T)$ a polynomial with constant term $1$ and coprime with $P_2(T)$. 

\item Over  $\kappa(\mc {E}_{\alpha})$ we can decompose $M=M_1\oplus M_2$. If $P_1^*(T)=T^{\mr{deg}(P_1)}P_1(T^{-1})$ then $P_1^*(U)$ is zero on $M_1$ and invertible on $M_2$. Moreover $\mr{rank}_A M_1 = \mr{deg}(P_1)$.
\end{itemize}

We define several Hecke operators at $p$. In the symplectic case we follow the notation of \cite[\S 3.6]{hida_control}; let $\pfrak$ be a prime above $p$ in $F_0$ and $p$ and for $0\leq j \leq a-1$ define matrices $\alpha_j = \mr{diag}[1,\ldots,1,\overbrace{p,\ldots,p}^{j}] \in \mr{GL}_a(\mc O_{F_{0,\pfrak}})$. For $1 \leq j \leq a$ we define $\beta_j = \vierkant{\alpha_j}{0}{0}{w_0p^2\alpha_j^{-1}w_0}$ and $\beta_0 = \vierkant{\alpha_0}{0}{0}{p\alpha_0}$. 
If $U_{a,a}$ denotes the unipotent of the parabolic $P_{a,a}$, we define 
\begin{align*}
\U_{\pfrak,j} = \frac{[U_{a,a}(\mc O_{F_{0,\pfrak}})\beta_j U_{a,a}(\mc O_{F_{0,\pfrak}})]}{p^{d_{\pfrak}(a-j)(a+1)}}
\end{align*} 
 for $d_{\pfrak}$ the degree of the extension $F_{0,\pfrak}/\Q_p$. We let $\m U_{G,p}=\bigotimes_{\pfrak}\Z_p[U_{\pfrak,j}]_j$ be the Hecke algebra generated by these operators and define $\U_{G,p}=\prod_{\pfrak}\prod_{j=0}^{a-1}\U_{\pfrak,j}$.

We define similarly $\m U_{G,p}$ and $U_{G,p}$ for $G$ unitary following \cite[\S 6]{hida_control}.

All these Hecke operators act naturally on the space of families of automorphic forms, respect the integral structure and the filtration given by the corank. If one prefers the analytic formulation of Hecke operators, for classical weights $\kappa$ one have to multiply the double coset action on forms by $\kappa(\alpha_j)^{-1}$. It is well-known that $\U_{G,p}$ is completely continuous on the space of overconvergent forms and families \cite[\S 4.2]{pel}

We define now the Hecke algebra. Let $\mathfrak{l}$ be a prime ideal of $F_0$ above $l \neq p$, we define \begin{align*}
\m T_{G,\mathfrak{l}} = \Z_p[ G(\mc O_{F_{0,\mathfrak{l}}}) \setminus G(F_{0,\mathfrak{l}})/G(\mc O_{F_{0,\mathfrak{l}}})  ].
\end{align*}
When $G$ is symplectic, $\m T_{G,\mathfrak{l}}$ is generated by the image of diagonal matrices
\[
[\varpi_{\mathfrak{l}},\ldots,\varpi_{\mathfrak{l}},1,\dots,1] \in \GL_a(F_{0,\mathfrak{l}})
\]
which are embedded in $\GSp_{2a}$ as before. 

If $G$ is unitary and $\mathfrak{l}$ is split in $F$, then $G(F_{0,\mathfrak{l}})\cong \GL_{a+b}(F_{0,\mathfrak{l}})\times \m G_m(\Q_l) $ and $\m T_{G,\mathfrak{l}}$ is generated by the same matrices. If $\mathfrak{l}$ is inert, $G(F_{0,\mathfrak{l}})$ is contained in $\GL_{a+b}(F_{\mathfrak{l}})$ and generated by the same diagonal matrices.

Let $N$ be a prime-to-$p$ integer containing all prime numbers which are norms of prime ramified in $F$ or for which $G$ is not quasi-split. The abstract Hecke algebra of prime-to-$Np$ level is then \begin{align*}
 \m T_G^{(Np)} = \otimes'_{\mathfrak{l} \nmid Np} \m T_{G,\mathfrak{l}}.
\end{align*}

It naturally acts on the space of overconvergent forms and families as defined in \cite[\S 4.1]{pel}.

Let $q \leq s$. We shall denote by $\mathcal{E}_{a,s}^q(\underline v,\underline w)$ the eigenvariety associated with:
\begin{itemize}
\item for $\mc U= \mathrm{Spm}(A)$ we choose  an open affinoid inside $\mc W_a(\underline w)^{s}$ or $\mc W_{b,a}(\underline w)^{s}$,
\item for $M$, if $\mc H$ is neat we choose 
\[
\M_{a,s}^q\colonequals \Homol^0(\mathfrak S_{G}^{\ast,\rig}(\mc H)(\underline v) \times \Spm(A),\eta_{A,\ast} \underline \omega_{\underline v,\underline w}^{\dagger \chi_{\mc U}^{\un}} \otimes \mc I^q).
\]
Otherwise, we choose $\mc H' \subset \mc H$ with $\mc H'$ neat and we take 
\begin{align*}
\M_{a,s}^q\colonequals {\Homol^0(\mathfrak S_{G}^{\ast,\rig}(\mc H')(\underline v) \times \Spm(A),\eta_{A,\ast} \underline \omega_{\underline v,\underline w}^{\dagger \chi_{\mc U}^{\un}} \otimes \mc I^q)}^{\mc H / \mc H'}.
\end{align*}
Note that the notation is slightly different from above, but we prefer to stress in this section the genus $a$ and the corank $s$ of the weights of our families of modular forms.
\item for $\mathbf{T}$ we choose $\m T_{a,s}^{q}\colonequals \mr{Im}(\m T_{G}^{(Np)}\otimes{\m U_{G,p}} \rightarrow \mr{End}_{A}(\m M_{a,s}^q))$,
\item $U=\U_{G,p}$.
\end{itemize}

Before we can use Buzzard's machinery we need the following proposition.
\begin{prop}\label{moduliPR}
The module $\M_{a,s}^q$ is $\mr{Pr}$ over $\mc U$.
\end{prop}
\begin{proof}
If $\mc H$ is neat we can apply \cite[Lemma 2.11]{buzz_eigen} to the exact sequences in Theorem~\ref{teo: fund ex sec}. Otherwise, it is enough to remark that $\M_{a,s}^q$ is a direct summand of ${\Homol^0(\mathfrak S_{G}^{\ast,\rig}(\mc H')(\underline v) \times \Spm(A),\eta_{A,\ast} \underline \omega_{\underline v,\underline w}^{\dagger \chi_{\mc U}^{\un}} \otimes \mc I^q)}$ (the action of the finite group  ${\mc H / \mc H'}$ is always diagonalisable in characteristic zero) which is $\mr{Pr}$ by the same argument as above.
\end{proof}
For each fixed $\underline w$, we have the eigenvarieties 
\begin{align*}
{\mc E}_{a,s}^q(\underline v,\underline w) \rightarrow \mc Z_{a,s}^q(\underline v,\underline w)
\end{align*} which are independent of $\underline v$ if $\underline v$ is small enough \cite[Lemma 5.6]{buzz_eigen}. Letting $\underline w$ go to infinity we can glue the different eigenvarieties \cite[Lemma 5.5]{buzz_eigen}.
\begin{defi}
The eigenvariety $\mathcal{E}_{a,s}^q$ for forms of corank at most $q$ over the weight space $\mc W_a^{s}$ or $\mc W_{b,a}^{s}$ of weights of corank at least $s$ is defined as
\begin{align*}
\mathcal{E}_{a,s}^q\colonequals \varinjlim_{v,w} \mathcal{E}_{a,s}^q(\underline v,\underline w).
\end{align*}
\end{defi}
We say that a point $x$ in $\mathcal{E}_{a,s}^q$ is classical if the system of eigenvalues associated with $x$ appears in the space of classical forms $\M_k(\mc H,\overline{\Z})$.
\begin{prop}
Classical points are Zariski dense in $\mathcal{E}_{a,s}^q$.
\end{prop}
\begin{proof}
Indeed from \cite[Theorem 6.7]{pel} we know that forms of small slope (w.r.t. the weight) are classical. The points satisfying the condition of {\it loc. cit.} are clearly Zariski dense in $\mc W_a$; we can then proceed as in \cite[Theorem 5.4.4]{urban_eigen}. 
\end{proof}
\subsection{Relations between different eigenvarieties} \label{subsec: relations}
We now want to analyse the relations between $\mathcal{E}_{a,s}^q$ when varying $a$, $q$ or $s$. We begin with a lemma;
\begin{lemma}\label{lemmasum}
Suppose that $M$ is an extension of two potentially ON-module $M_1$ and $M_2$ over $A$. Suppose that $M_1$ and $M_2$ are $U$-stable, then 
\begin{align*}
\det(1 -TU \vert M) = \det(1 -TU \vert M_1)\det(1 -TU \vert M_1).
\end{align*}
\end{lemma}
\subsubsection*{Changing $q$ and $s$} We begin by letting $q$ vary. Let $q' < q \leq s $, we have a natural injection $ \M_{a,s}^{q'} \rightarrow  \M_{a,s}^q$ which induces by restriction a surjective map $ \m T_{a,s}^{q} \rightarrow \m T_{a,s}^{q'}$. In particular this gives us a closed immersion $\mathcal{E}_{a,s}^{q'} \rightarrow \mathcal{E}_{a,s}^q$.

We now vary $s$ too. Let $s' < s$ and $\mc U'$ be an open affinoid of $\mc W_a(\underline w)^{s'}$. (Everything is the same if $\mc U'$ in $\mc W_{b,a}(\underline w)^{s'}$.) Let $\mc U = \mc U' \times_{W_a(\underline w)^{s'}} \mc W_a(\underline w)^{s}$ and denote by $A$ resp. $A'$ the affinoid algebra for $\mc U$ resp. $\mc U'$. By definition of the sheaves $\omega^{\dagger\chi_{\mc U}^{\mr{un}}}_{\underline v,\underline w}$ and $\omega^{\dagger\chi_{\mc U'}^{\mr{un}}}_{\underline v,\underline w}$ we have 
\begin{align*}
\M_{a,s'}^q \otimes_{A'} A =\M_{a,s}^q. 
\end{align*}
This implies immediately that $\mathcal{E}_{a,s'}^{q} \times_{\mc W_{a}^{s'}} \mc W_{a}^{s} $ and $\mathcal{E}_{a,s}^{q}$ have the same closed points. This also implies that $\mc Z_{a,s}^q$ is the base change of $\mc Z_{a,s'}^q$ because the characteristic series for $\U_p$ is stable under base-change \cite[Lemma 2.13]{buzz_eigen}. But this does \emph{not} imply that $\mathcal{E}_{a,s'}^{q} \times_{\mc W_{a}^{s'}} \mc W_{a}^{s} $ and $\mathcal{E}_{a,s}^{q}$ are isomorphic, as non-reducedness issues could appear. Still, we can prove the following:
\begin{teo}\label{teo:glue}
The reduced eigenvarieties $\mathcal{E}_{a,q}^{q,\mr{red}}$ ($q=0,\ldots, a $) glue into a (non-equidimensional) eigenvariety $\mathcal{E}_{a}$ over $\mc W_a$ (resp. $\mc W_{a,b}$).
\end{teo}
To prove the theorem, we need to study the relation between our (reduced) eigenvarieties and other non-equidimensional eigenvarieties. In \cite{Hansen}, Hansen constructs a non necessarily equidimensional eigenvariety $\mathbb{E}:=\mathbb{E}_{\mc H}$ over $\mc W_{a}$ (resp. $\mc W_{a,b}$) starting from the method of Ash--Stevens. (A similar results, using Urban's construction, has been obtained by \cite{Xiang}.) 

The main theorem of this section will be easily deduced from the following:
\begin{teo}
We have a closed embedding $\iota_{\mr{H},q}:\mathcal{E}_{a,q}^{q,\mr{red}} \rightarrow \mathbb{E}$ compatible with the structural morphisms to the weights spaces and the closed immersions $\mc W_a^{q} \rightarrow \mc W_{a}$ (resp. $\mc W_{b,a}^{q} \rightarrow \mc W_{b,a}$).
\end{teo} 
\begin{proof}
Recall that the degeneracy of the BGG complex for $S^{\tor}_G(\mc H)$, proved in \cite{FaltingsChai, LanPolo}, gives us a Hecke equivariant injection 
\begin{align}\label{ESinj}
\Homol^0(S^{\tor}_G(\mc H),\omega^k) \hookrightarrow \Homol^d(S^{\tor}_G(\mc H), V_{\tilde{k}})
\end{align}
where $d$ is the length of the longest element  $w_0$ in the Weyl group, $\tilde{k}$ is $w_0(k+\rho)-\rho$ (for $\rho$ the half-sum of the positive roots) and $V_{\tilde{k}}$  is the local system associated with the representation of highest weight $\tilde{k}$.\\
To prove this theorem, it is enough to show that we can apply \cite[Theorem 5.1.6]{Hansen}. We take as $\mathfrak{D}_1$ the eigenvariety datum for $\mathcal{E}_{a,q}^{q}$ and as $\mathfrak{D}_2$ the one of \cite[Definition 4.3.2]{Hansen}. Hansen's uses the cohomology of certain complexes of overconvergent distributions, interpolating the cohomology of the local system $V_{\tilde{k}}$ over the weight space.\\
We need additional data (i)--(iii) to apply his theorem: we take the natural immersion $\mc W_a^{q} \rightarrow \mc W_{a}$ (resp. $\mc W_{b,a}^{q} \rightarrow \mc W_{b,a}$) for (i) and for (ii) the identity map of abstract Hecke algebra. 
For (iii), we define $\mathfrak{Z}^{\mr{cl}}$ the set of points $(\kappa,\alpha)$ of $\mathcal{Z}$ with $\kappa$ classical dominant and $v_p(\alpha)$ satisfies the bounds of \cite[Theorem 6.7]{pel} and \cite[Theorem 3.2.5]{Hansen}. This set is very Zariski dense. 

This ensures us that generalised eigenspace $\Homol^0(\mathfrak S^{\ast, \rig}_G(\mc H)(\underline v) \times \Spm(A), \eta_\ast\underline \omega_{\underline v,\underline w}^{\dagger\kappa}\otimes \mc I^q)[\U_{G,p}=\alpha^{-1}]$ consists of classical forms which injects in $\Homol^0(S^{\tor}_G(\mc H),\omega^{\kappa})[\U_{G,p}=\alpha^{-1}]$, and similarly for $\Homol^*(S^{\tor}_G(\mc H),V_{\tilde{k}})$ and its overconvergent counterpart. The required divisibility for the characteristic polynomial of the Hecke operators is implied by the Eichler--Shimura injection \ref{ESinj}. 
\end{proof}
\begin{proof}[Proof of Theorem \ref{teo:glue}]
The previous theorem tells us that $\bigcup_q \iota_{\mr{H},q}(\mathcal{E}_{a,q}^{q,\mr{red}})$ is a closed subspace of  $\mathbb{E}$. We define $\mathcal{E}_{a}$ to be this closed subvariety with its reduced structure.
\end{proof}

\subsubsection*{Changing $a$}
We are now interested in studying the complement of $\mathcal{E}_{a,s}^0$ in $\mathcal{E}_{a,q}^q$. Recall from Section \ref{subsec: p-adic Siegel morphism} that we have a surjective morphism:
\begin{align*}
\M_{a,s}^q \rightarrow \bigoplus_{C_1(\mc H p^n)} \M_{a-1,s-1}^{q-1} \:\;\;\;\; ( s > 1 ); \\
\M_{a,1}^1 \rightarrow \bigoplus_{C_1(\mc H p^n)} \M_{a-1,0}^{0}\otimes_{B} A \:\;\;\;\; ( s=1 ).
\end{align*}
\begin{prop}
Let $G=\GSp_{2a}$ and $G'=\GSp_{2a-2}$.The above morphisms induce a surjective map between $\m T_G^{(Np)} \otimes_{\Z_p} \Q_p$ and  $\m T_{G'}^{(Np)} \otimes_{\Z_p} \Q_p$. A similar result holds for $\GU_{a,b}$ and $\GU_{a-1,b-1}$.
\end{prop}
\begin{proof}
For ${\mr {GSp}_{2a}}_{/\Q}$ this is proven in \cite[Korollar 1]{HeckeSiegel} and for a general totally real field the proof is exactly the same. For unitary group the proof is similar. 
\end{proof}

\begin{prop}
Let $G=\mr{GSp}_{2a,/F_0}$ and $G'=\mr{GSp}_{2a-2/F_0}$. The Siegel morphism sends  $\U_{\pfrak,a}$ for $G$ to $\U_{\pfrak,a-1} $ for $G'$ and, for $0\leq j \leq a-1$, $\U_{\pfrak,j}$ to $p^{d_{\pfrak}j}\U_{\pfrak,j}$.
\end{prop}
\begin{proof}
We start with the symplectic case; let $m=a-j$, $M_{m \times j}(\mc O_{F_{0,\pfrak}})$ the set of matrices of size $m$ times $j$ with entries in $\mc O_{F_{0,\pfrak}}$ and $S_a(\mc O_{F_{0,\pfrak}})$ the set of $a$ times $a$ matrices with entries in $\mc O_{F_{0,\pfrak}}$ such that $w_0xw_0=^tx$. These are matrices symmetric w.r.t. the reflection along the anti-diagonal.

We have an explicit decomposition of the double coset  \cite[Proposition 3.5]{hida_control}:
\begin{align*}
U_{a,a}(\mc O_{F_{0,\pfrak}})\beta_j U_{a,a}(\mc O_{F_{0,\pfrak}}) = \bigsqcup_{u,x} U_{a,a}(\mc O_{F_{0,\pfrak}}) \beta_j U_u U_x
\end{align*}
where
\[
U_x = \vierkant{1}{x}{0}{1}, x \in {S_{a} (\mc O_{F_0,\pfrak})}/p\alpha_j^{-1}S_a(\mc O_{F_0,\pfrak})w_0 p\alpha_j^{-1} w_0
\]
and $U_u =\vierkant{V_u}{0}{0}{w_0 V_u w_0}$ with $V_u = \vierkant{1}{u}{0}{1}$ for $u \in M_{m \times j}(\mc O_{F_{0,\pfrak}}/p)$. Note that this latter is a set of representatives for $U_{\GL_a}(\mc O_{F_{0,\pfrak}})p\alpha_j^{-1}U_{\GL_a}(\mc O_{F_{0,\pfrak}})$.

As the Siegel morphism is equivariant for the action of $P_{a,1}$, we have
\begin{align*}
\Phi \left( F \vert 
\left( 
\begin{array}{cccc}
a_1 & a_2 & b_1 & b_2 \\
0   & a_3 & b_3 & b_4 \\
0   & c_3 & d_3 & d_2 \\
0   & 0   & 0   & d_1
\end{array}
\right) \right) = \Phi(F) \vert \left( 
\begin{array}{cc}
a_3 & b_3 \\
c_3 & d_3 
\end{array}
\right).
\end{align*}
We have that $\beta_j$ for $G$ is sent to $\beta_j$ for $G'$. Every $V_{u'} = \vierkant{1}{u}{0}{1}$ for $u' \in M_{(m-1) \times j}(\mc O_{F_{0,\pfrak}}/p)$ has exactly $p^{jd_{\pfrak}}$ preimages between the matrices $V_u = \vierkant{1}{u}{0}{1}$ with $u \in M_{m \times j}(\mc O_{F_{0,\pfrak}}/p)$. Each $x$ in the decomposition for $G'$ has exactly
\[
p^{d_{\pfrak}(2a-j)}=p^{d_{\pfrak}(a-j)(a+1)}/p^{d_{\pfrak}a(a-1-j)}
\]
counterimages which is the factor by which we divide in the definition of $\U_{\pfrak,j}$.
The unitary case is similar and left to the reader.
\end{proof}
If $F_0=\Q$, this has been proved in \cite[Theorem 1.1]{Dickson}.

Summing up, the two previous propositions and Lemma \ref{lemmasum} allow us to use \cite[Theorem 5.1.6]{Hansen} to obtain a closed immersion of $\mathcal{E}_{a-1,s-1}^{q-1}$ into $\mathcal{E}_{a,s}^q$. Note that $\U_{G,p}$ is not sent into $\U_{G',p}$; for example, in the symplectic case it is mapped to $p^{\frac{[F_0:\Q]a(a-1)}{2}}\U_{G',p}\prod_{\pfrak}\U_{\pfrak,a-1}$. (The construction of the spectral variety it is indeed not canonical, as it depends on the choice of a compact operator.) 


To conclude, recall that in Proposition \ref{moduliPR} we have seen that the Siegel morphism splits as morphism of Banach modules; it is then natural to ask the following. 
\begin{quest}
Can one choose this splitting to be Hecke equivariant? 
\end{quest}
A positive answer would not imply that the eigenvariety is disconnected but would hint to the fact that one should be able to define directly families of non-cuspidal forms inducing families of cusp forms from parabolic subgroups of $G$. Some instances of this parabolic induction for ordinary forms has been proven by \cite{skinn_urb} from $\GL_2$ to $\GU(2,2)$ but using a pullback from from $\GU(3,3)$. A direct construction not involving pullback formulas would be more interesting and of much more general use.
\begin{remark}
We want to point out that for each parabolic $P$ of $G$ we can define a $P$-ordinary projector. If $M$ is the Levi of $P$ and $\pi$ a cuspidal automorphic representation of $M$, there is, for weight big enough, a unique Eisenstein series $E(\pi)$ which is $P$-ordinary.
\end{remark}

It is known that for $\GL_2$ the critical $p$-stabilisation of a level $1$ Eisenstein series is a $p$-adic cusp form. A similar phenomenon appears also in higher genus.

\subsection*{An example} We consider now to the case of $\U(2,2)$ for $F/\Q$ a CM extension. 
We choose  $f_0 \in \rm S_k(\G_1(N))$ a cusp form which is ordinary at $p$. Let $f$ be its non-ordinary $p$-stabilisation. Suppose that $f$ is $\theta$-critical ({\it i.e.} it is in the image of the $p$-adic Maa\ss{}--Shimura operator $\theta^{k}$); we suppose that at the point on the eigencurve $\mc C$ corresponding   to $f$ there is a unique family $F$ passing through $f$. This is the case if $f$ is CM (which is conjecturally always the case): a deep results of Bella\"iche \cite[Theorem 2.16]{bellaiche} ensure us that the the eigencurve is smooth at $f$. Note that it is known that the structural morphism $\kappa: \mc C \rightarrow \mc W$ is not \'etale at this point.

Suppose now that we can define the $p$-adic Klingen--Eisenstein series $E(F)$ interpolating the classical Klingen--Eisenstein series as done in the ordinary case in \cite[Theorem 12.10]{skinn_urb}. We know \cite[(11.64)]{skinn_urb} that at a classical point $x$ the constant term at genus one cusp labels  is a suitably normalised multiple of $L(f_x,k-1)f_x$, being $f_x$ the form corresponding to $x$. Hence,  the constant term is divisible by the two-variable $p$-adic $L$-function for $F$ evaluated at $\kappa(x)-1$. We know that $L_p(f, j)=0$ for all $0\leq j \leq \kappa(x)-1$ \cite[Theorem 2]{bellaiche}. In particular, the generically non cuspidal $E(F)$ at the point corresponding to  $E(f)$ would degenerate to a cuspidal form. It is an interesting question to understand how the geometry of smaller eigenvarieties (in our case, the non-\'etalness of $\kappa$) influences the geometry of bigger eigenvarieties (in our case, the non-cuspidal eigenvariety for $\U(2,2)$).

\subsection{On a conjecture of Urban}
In \cite{Hansen}, Hansen constructs a non necessarily equidimensional eigenvariety $\mathbb{E}_{\mc H}$ using ideas of Ash--Stevens. A similar results, using Urban's construction, has been obtained by \cite{Xiang}. Their two constructions are different from the machinery of Buzzard (that we have used in this paper) and it is not clear which is the dimension of an irreducible component of these eigenvarieties. In \cite[Conjecture 5.7.3]{urban_eigen}, Urban made the following precise conjecture for the dimension of irreducible components.
\begin{conj}
Let $x$ be a point belonging to exactly one irreducible component of $\mathbb{E}_{\mc H}$ and let $\theta$ be the corresponding system of eigenvalues.  Define $d$ to be the number of consecutive cohomology degrees in which the system $\theta$ appears. Then the image of the irreducible components to which $x$ belongs in the weight space is of codimension $d-1$.
\end{conj}
Urban has shown this conjecture for the cuspidal irreducible components. In general the number $d$ is not very easy to compute. Using work of Harder \cite{Harder}, we calculate this number for $\mathrm{GSp}_4$. Under mild hypothesis on the relation between irreducible components of $\mathbb{E}_{\mc H}$ and $\mc E_{2}$ we shall be able to prove Urban conjecture.

Let $S$ be the Siegel variety for ${\GSp_4(\Z)}$, denote by $S^{\mr{BS}}$ its Borel--Serre compactification, $\iota : S \rightarrow S^{\mr{BS}}$ the open immersion and by $\partial S$ the boundary of this map.  Let $(k_1,k_2)$ be the highest weight of a non zero algebraic representation of $\mr{Sp}_4$ and $\mc M_{k_1,k_2}$ the corresponding local system on $S$. Recall that we have the long exact sequence of cohomology:
\begin{gather*}
\cdots \rightarrow \Homol^{\bullet-1}(\partial S,\mc M_{k_1,k_2})\stackrel{\delta^{\bullet-1}}{\rightarrow} \\
\Homol_c^{\bullet}( S,\mc M_{k_1,k_2})\stackrel{i^{\bullet}}{\rightarrow} \Homol^{\bullet}( S,\mc M_{k_1,k_2})\stackrel{r^{\bullet}}{\rightarrow} \Homol^{\bullet}(\partial S,\mc M_{k_1,k_2})\rightarrow \cdots,
\end{gather*}
where $\Homol_c^{\bullet}(S,\mc M)= \Homol^{\bullet}(S,\iota_{!}\mc M)$. The cohomology of the boundary splits as 
\begin{align*}
\Homol^{\bullet}(\partial S,\mc M_{k_1,k_2}) = r^{\bullet}(\Homol^{\bullet}( S,\mc M_{k_1,k_2})) \oplus \Ker(\delta^{\bullet});
\end{align*}
we shall call the first term the {\it Eisenstein cohomology} and by the second term the {\it compactly supported Eisenstein cohomology}. We define also the interior cohomology: \begin{align*}
\Homol_{!}^{\bullet}( S,\mc M_{k_1,k_2})=\mr{Im}(i^{\bullet}). 
\end{align*}
Faltings--Chai proved an Eichler--Shimura morphism for Siegel forms:
\begin{align*}
\Homol^0(S,\underline \omega^{k_1,k_2}) \hookrightarrow \Homol^3(S, \mr{Sym}^{k_1-k_2}(\Q^2)\otimes \mr{det}(\Q^2)^{k_2-3}).
\end{align*}
Let $x$ be a system of eigenvalues for the Hecke algebra acting on $\Homol^{\bullet}( S,\mc M_{k_1,k_2})$; we shall say that $x$ is {\it Eisenstein} if $x$ corresponds to a system of eigenvalues in $\mc E_{2}$ and the $x$-eigenspace $\Homol^{\bullet}( \partial S,\mc M_{k_1,k_2})[x] \neq 0$. 

The classification of the boundary cohomology in terms of the weight has been given in \cite{Harder} using the spectral sequence coming from the stratification of $\partial S$. In the notation of \cite{Harder} we have $n_1= k_1-k_2 $ and $n_2=k_2-3$ and we define   $p(k)=\dagger$  if $n_2$ is odd and $p(k)=\star$ otherwise.
From \cite[page 156]{Harder} we know in which degree  the cohomology $\Homol^{\bullet}( \partial S,\mc M_{k_1,k_2})$ is concentrated:
\begin{itemize}
\item if $p(k)=\dagger$ and $n_1 \neq 0$, the boundary cohomology vanishes;
\item if $p(k)=\dagger$ and $n_1 = 0$ the boundary cohomology is concentrated in degrees $2$ and $3$;
\item if $p(k)=\star$ and $k_1 > k_2 > 0$ the boundary cohomology vanishes;
\item if $p(k)=\star$ and $k_1=k_2 \neq 0$ the boundary cohomology vanishes;
\item if $p(k)=\star$ and $ n_2=0$ but $n_1 >0$ than we have cohomology in degrees $2$ and $3$;
\item if $p(k)=\star$ and $n_1=n_2=0$ the cohomology is in degrees $0$, $2$, $3$, and $5$.
\end{itemize}
\begin{teo}
Let $\mc I$ be an irreducible component of $\mc E_{2}$; suppose there exists a classical point $x \in \mc I(\C_p)$ corresponding either to a Klingen Eisenstein series or a Siegel Eisenstein series of weight different from $(3,3)$. Then Urban's conjecture holds for all components $\mathbb{I}$ of $\mathbb{E}_{\GSp_4(\Z)}$ which satisfies $\mathbb{I}^{\mr{red}}\cong \mc I^{\mr{red}}$.
\end{teo}
\begin{remark}
We need to assume that an irreducible component of our eigenvariety is an irreducible component of $\mathbb{E}_{\GSp_4(\Z)}$ as this is not a priori clear.
\end{remark}
\begin{proof}
We check case by case the previous list.

Suppose that $k_1 > k_2$; our point is then in $\mc E_{2,1}^1$ and there are no consecutive cohomology group where this system appears. Urban's conjecture says that the irreducible component containing $x$ has the same dimension as the total weight space, which is indeed the case.

Let now $p(k)=\dagger$, $n_1=0$ and $k_2 \neq 3$. We refer to \cite[\S 2.5]{Harder}. We have boundary cohomology only if $k_2$ is even; in this case all, the filtered pieces of $\Homol^2( \partial S,\mc M_{k_1,k_2}))$ corresponds to the $\Homol^1$  with coefficient $\Sym^{k_2-2}$ (if we are looking at the Klingen parabolic) or $\mr{Sym}^{2k_2-2}$ (Siegel parabolic) and the $\Homol^0$ for the same sheaves (which vanishes). \\
The filtered pieces of $\Homol^3(\partial S,\mc M_{k_1,k_2}))$ are the $\Homol^1$ of the modular curve for forms of weight $k_2$ (if we are looking at the Klingen parabolic) or $2k_2-4$ (Siegel parabolic) and the $\Homol^0$ of modular curves with coefficients $\C$ and $k_2-3$. 

Suppose now that $x$ belongs to $\mc E_{2,1}^1$ so that it correspond to a system of eigenvalues of a cusp form. This case is studied in \cite[\S 2.5.1]{Harder} and the only contribution  to $\Homol_c^{\bullet}( S,\mc M_{k_1,k_2})$ is in degree $3$. Then $x$ contributes only to $\Homol^{3}(S,\mc M_{k_1,k_2})$ by Poincaré duality. The dimension of $\mc I$ is the conjectured one, namely $2$.

Suppose now that $x$ belongs to $\mc E_{2,2}^2$ but not to $\mc E_{2,2}^1$, so it is a classical parallel weight Siegel Eisenstein series. Summing up \cite[\S 2.5.1-4]{Harder}, as the weight is not $(3,3)$, we have contribution in two consecutive degrees ($2$ and $3$).  In this case Urban's conjecture predicts the codimension to be $1$ which is indeed the case.
\end{proof}
\bibliographystyle{amsalpha}
\bibliography{biblio}
\end{document}